\documentclass[11pt]{article}
\usepackage{amsfonts,amsmath,latexsym}
\usepackage{amssymb}
\usepackage{mathtools}
 \usepackage{enumerate}
\DeclareMathOperator{\const}{const}

\def\R{\mathbb R}
\def\N{\mathbb N}

\def\E{\mathbb E}

\def\shb{{\cal B}}
\def\shc{{\cal C}}
\def\shd{{\cal D}}
\def\shf{{\cal F}}

\def\shp{{\cal P}}
\def\shs{{\cal S}}
\def\shy{{\cal Y}}
\def\sign{\text{sign}}
\newcommand{\norm}[1]{\left\| #1 \right\|}
\newcommand{\bracket}[1]{\langle #1\rangle}

\newtheorem{theo}{Theorem}[section]
\newtheorem{lemma}[theo]{Lemma}
\newtheorem{Assumption}[theo]{Assumption}
\newtheorem{hypo}[theo]{Hypothesis}
\newtheorem{prop}[theo]{Proposition}
\newtheorem{rem}[theo]{Remark}
\newtheorem{cor}[theo]{Corollary}
\newtheorem{defi}[theo]{Definition}

\newcommand{\beqnar}{\begin{eqnarray*}}
\newcommand{\eeqnar}{\end{eqnarray*}}
\newcommand{\ba}{\begin{array}}
\newcommand{\ea}{\end{array}}

\newenvironment{proof}[1]{\begin{trivlist}\item {\it
\bf Proof.}\quad} {\qed\end{trivlist}}
\newenvironment{prooff}[1]{\begin{trivlist}\item {\it
\bf Proof}\quad} {\qed\end{trivlist}}

\newcommand{\qed}{\nopagebreak\hspace*{\fill}
{\vrule width6pt height6ptdepth0pt}\par}

\begin{document}

\title{Probabilistic representation for solutions
of an irregular porous media type  equation:
the  degenerate case.}

\author{ Viorel Barbu (1), Michael R\"ockner (2)
and Francesco Russo (3) } 

\date{}
\maketitle

\thispagestyle{myheadings}
\markright{Irregular degenerate porous media type equation}

{\bf Summary:} We consider a possibly degenerate 
porous media type equation over all of $\R^d$ with $d = 1$,
 with monotone  discontinuous coefficients with linear growth and prove a
probabilistic representation of its solution in terms of an associated 
microscopic diffusion. 
This equation is motivated by some singular
behaviour arising in complex self-organized critical systems.
The main idea consists in approximating the equation by equations
 with monotone non-degenerate coefficients and
deriving some new analytical properties
of the solution.

{\bf Key words}: singular degenerate porous media type equation,
probabilistic representation.

{\bf2000  AMS-classification}: 60H30, 60H10, 60G46,
 35C99, 58J65


{\bf Actual version:} August 18th 2009

\begin{itemize}
\item[(1)] Viorel Barbu,
University A1.I. Cuza, Ro--6600 Iasi,
Romania.
\item[(2)] Michael R\"ockner,
Fakult\"at f\"ur Mathematik, 
Universit\"at   Bielefeld, 
\\ D--33615 Bielefeld, Germany  and 
Department of Mathematics and Statistics, 
Purdue University, 
W. Lafayette, IN 47907, USA.
\item[(3)] Francesco Russo,
INRIA Rocquencourt, Equipe MathFi and Cermics
Ecole des Ponts,
Domaine de Voluceau,
Rocquencourt - B.P. 105,
F-78153 Le Chesnay Cedex, France\\ and
Universit\'{e} Paris 13, 
Institut Galil\'{e}e, Math\'ematiques,
99, avenue J.B.~Cl\'{e}ment,
F-93430 Villetaneuse,
France
\end{itemize}

\vfill \eject

\section{Introduction}

\setcounter{equation}{0}

We are interested in the probabilistic representation of the solution to
 a porous media type equation  given by
\begin{equation}
\label{PME}
\left \{
\begin{array}{ccc}
\partial_t u&=& \frac{1}{2} \partial_{xx}^2(\beta(u)),  \ t \in [0, \infty[
\\
u(0,x)& = & u_0(x), \ x \in \R,
\end{array}
\right.
\end{equation}
in the sense of distributions, where $u_0$ is an initial bounded 
probability density.
We look for a solution of (\ref{E1.0}) with time evolution in $L^1(\R)$.

We make the following assumption.
\begin{Assumption}\label{E1.0}
\begin{itemize}
\item  $\beta: \R \rightarrow \R$ is
monotone increasing. 
\item 
$\vert \beta (u) \vert \le  {\rm const} \vert u \vert, \ u \ge 0.  $ \\
In particular,  $\beta$ is right-continuous at zero and $\beta(0) = 0$.
\item  There is  $\lambda > 0$ such that
$(\beta +  \lambda id)(x) \rightarrow \mp \infty$ when $x \rightarrow 
\mp \infty$.
\end{itemize}
\end{Assumption}
\begin{rem}\label{Rint} 
\begin{description}
\item{(i)} By one of the consequences of our main result, see Remark
\ref{Pos} below,  the solution to \eqref{PME} is non-negative, since 
$u_0 \ge 0$.
Therefore, it is enough to assume that only the restriction of $\beta$ 
to  $\R_+ $ is increasing such that $\vert \beta(u) \vert \le {\rm const} 
\vert u \vert$ 
for $u \ge 0$, and 
$(\beta +  \lambda id)(x) \rightarrow \infty$ when $x \rightarrow +\infty$.
Otherwise, we can just replace $\beta$ by an extension of 
 the restriction of  $\beta$ to $\R_+$ which satisfies Assumption 
\ref{E1.0}, e.g. take its
 odd symmetric extension.
\item{(ii)} In the main body of the paper, we shall in fact replace $\beta$ 
 with the ''filled'' associated graph, see remarks after Definition 
\ref{D2.4b} for details; in this way, we consider $\beta$ as a multivalued 
function and Assumption \ref{E1.0} will be replaced by Hypothesis \ref{H3.0}.  
\end{description}
\end{rem}
Since $\beta$ is monotone, (\ref{E1.0}) implies
   $ \beta (u) = \Phi^2(u) u, \ u  \ge 0$, $\Phi$ being a
non-negative bounded Borel function. 
We recall that when $\beta (u) = \vert u \vert
 u^{m-1}$, $m  >   1$, (\ref{PME}) is  nothing else but the
classical {\it porous media equation}.

One of our targets is to consider $\Phi$
 as continuous except for a possible jump at one positive point,
say $e_c  > 0  $. A  typical example is 
\begin{equation} \label{PMEa} 
\Phi (u) = H(u-e_c), 
\end{equation}
$H$ being the Heaviside
function. 

The analysis  of (\ref{PME}) and its probabilistic representation
 can be done in the framework of monotone partial differential
equations (PDE) allowing multi-valued coefficients and will be discussed
in detail in the main body of the paper. In this introduction, for simplicity, 
we restrict our presentation to the single-valued case.

\begin{defi} \label{DNond} 
\begin{itemize} 
\item We will say that equation (\ref{PME}) or $ \beta$ is 
{\bf non-degenerate} if on each compact,
 there is a constant $c_0 > 0$ such that
$ \Phi  \ge c_0  $.
\item We will say that equation (\ref{PME}) or $ \beta$ is 
{\bf degenerate} if  $ \lim_{u \rightarrow 0_+} \Phi(u) = 0  $
  in the sense
that for any sequence of non-negative reals $(x_n)$ converging
to zero, and $y_n \in \Phi(x_n)$ we have 
$\lim_{n \rightarrow \infty} y_n = 0$. 
\end{itemize} 
\end{defi}
\begin{rem}\label{DegNonD}
\begin{enumerate}
\item   $\beta$ may be in fact neither non-degenerate nor degenerate.
If $\beta$ is odd, which according to Remark
 \ref{Rint} (ii), we may always assume, then 
 $\beta$ is non-degenerate if and only if
$\liminf_{u \rightarrow 0+} \Phi(u) > 0$.
\item Of course, $\Phi$ in (\ref{PMEa}) is degenerate.
 In order to have $\Phi$ non-degenerate, one could  
add a positive constant to it.
\end{enumerate}
\end{rem}
Of course, $\Phi$ in (\ref{PMEa}) is degenerate.
 In order to have $\Phi$ non-degenerate, one could  
add a positive constant to it.

There are several contributions to 
the analytical study of (\ref{PME}),
 starting from
\cite{BeBrC75} for existence,   \cite{BrC79} for  uniqueness in the
case of bounded solutions 
 and \cite{BeC81} 
for continuous dependence on the coefficients.
The authors consider  the case where $\beta$ is  continuous, even
if their arguments allow some extensions for the discontinuous case.

As mentioned in the abstract,  the first motivation 
of this paper was to discuss  continuous time models 
 of self-organized criticality (SOC), which are described by equations
 of type    \eqref{PME} with $ \beta(u) = u \Phi^2(u)$ and $\Phi$ as in
 (\ref{PMEa}), see e.g.
 \cite{bak86} for a 
significant monography on the subject and the
 interesting physical papers 
 \cite{BanJa} and  \cite{clpvz}. 
For other comments related to SOC, one can read the introduction of \cite{BRR}.
The recent papers,  \cite{BDPR09, BBDRSoc}, discuss \eqref{PME} in the case 
\eqref{PMEa}, perturbed by a multiplicative noise.

The singular non-linear diffusion equation  (\ref{PME}) models
  the {\it macroscopic} phenomenon for
which we try to give a {\it microscopic}  probabilistic
representation, via a non-linear stochastic differential equation
(NLSDE)  modelling the evolution of 
a single point.

The most important contribution of \cite{BRR} was to establish
a probabilistic representation of (\ref{PME})
in the non-degenerate case.
 For the latter we established both existence and uniqueness.
In the degenerate case, even if the irregular diffusion equation
  (\ref{PME}) is well-posed, 
at that time, we could not prove existence  of solutions to the
corresponding NLSDE. This is now done in the present paper.

To the best of our knowledge the first author who considered a
probabilistic representation (of the type studied in this paper) for the
solutions of a non-linear deterministic PDE was McKean
\cite{mckean}, particularly in relation with the so called propagation of
chaos. In his case, however, the coefficients were smooth. From then on
the literature has  steadily grown and nowadays there is a vast amount of
contributions to the subject, especially when the non-linearity is in the
first order part, as e.g. in Burgers equation. We refer the reader to the
excellent survey papers \cite{sznit} and \cite{graham}.

A probabilistic interpretation of (\ref{PME}) when 
$\beta(u) = \vert u \vert  u^{m-1}, m   > 1,  $ was provided for
instance in \cite{BCRV}. For the same $\beta$, though  the method
could  be adapted  to the case where $\beta$ is Lipschitz, in
\cite{J00} the author  has studied the evolution equation  (\ref{PME})
when the initial condition and the evolution takes values
in the set of all  probability distribution functions on $\R$. 
Therefore, instead of an evolution  equation in $L^1(\R)$,  
he considers a state space of
functions vanishing at $- \infty$ and with value $1$ at $+ \infty$.
He  studies both the probabilistic representation
and propagation of chaos.

 Let us now describe the principle of the mentioned probabilistic
representation. 
The stochastic differential equation (in the weak sense)
rendering the probabilistic  representation is given 
by the following (random) non-linear diffusion:
\begin{equation}
\label{E1.2}
\left \{
\begin{array}{ccc}
Y_t &=& Y_0 + \int_0^t \Phi(u(s,Y_s)) dW_s  \\
{\rm Law \quad density } (Y_t) &=& u(t,\cdot), \\
\end{array}
\right.
\end{equation}
where $W$ is a classical Brownian motion.
The solution of that equation may be visualised as a continuous process $Y$
on some filtered probability 
space $(\Omega, \shf, (\shf_t)_{ t \ge 0}, P)$
equipped with a Brownian motion $W$.
  By looking  at a properly chosen version, we can
and shall assume that $Y:[0,T] \times \Omega \rightarrow \R_+ $ 
is $\shb([0,T]) \otimes \shf$-measurable. Of course, we can only
have (weak) uniqueness for (\ref{E1.2}) fixing the initial distribution, i.e. we have to fix the distribution (density)  $u_0$ of $Y_0$.

The connection with (\ref{PME}) is then given by the following
result, see also \cite{BRR}.

\begin{theo} \label{TI.1}
Let us assume the existence of a solution $Y$ for
(\ref{E1.2}).
Then $u: [0,T] \times \R \rightarrow \R_+$   provides a solution in the sense
of distributions of (\ref{PME}) with
$u_0 := u(0,\cdot)$.
\end{theo}

\begin{rem}\label{Pos}
\
An immediate consequence for the associated solution of (\ref{PME}) is
its positivity at any time if it starts with an initial value $u_0$
 which is positive.
Also the mass 1 of the initial condition  is conserved in this case.
However this property
follows already by approximation from Corollary 4.5 of \cite{BRR},
 which in turn is  based on the probabilistic  representation
in the  non-degenerate case, see  Corollary \ref{R4.11} below for details.

 \end{rem}
The main purpose of this paper is to show existence of the 
probabilistic representation equation (\ref{E1.2}),
in the case where $\beta$ is degenerate 
and not necessarily continuous. 
The uniqueness is only known if $\beta$ is non-degenerate and
 in some very special cases
in the degenerate case.

Let us now briefly and consecutively explain
the points that we are able to treat and the difficulties
 which naturally appear in the
probabilistic representation.

For simplicity we do this for
 $\beta$  being
single-valued (and) continuous.
However, with some technical complications this generalizes to
the multi-valued case, as spelt out in the subsequent sections.

\begin{enumerate} 
\item Monotonicity methods allow us to show existence and uniqueness
of solutions to  (\ref{PME}) in the sense of distributions
 under the assumption  that $\beta$ is monotone, that
there exists $\lambda > 0$ with $(\beta + \lambda id) (\R) =\R$ and
that $\beta$ is  continuous at zero, 
see Proposition 3.2 of \cite{BRR} and the references therein.
\item 
If $\beta$ is  non-degenerate, Theorem 4.3 of \cite{BRR},
allows to
construct a unique (weak)  solution $Y$ to the non-linear SDE in the first line
of (\ref{E1.2}), for any intial bounded probability density $u_0$ on $\R$.
\item 
Suppose $\beta$ to be degenerate. 
We fix a bounded probability density $u_0$.
We set $\beta_\varepsilon(u) = \beta(u) + \varepsilon u, \quad
\Phi_\varepsilon =  \sqrt {\Phi^2 + \varepsilon}$ and 
consider the weak solution  $ Y^\varepsilon$
of 
\begin{equation} \label{E1.2b}
Y^\varepsilon_t = Y^\varepsilon_0 + \int_0^t \Phi_\varepsilon (u^\varepsilon(s,Y^\varepsilon_s)) dW_s, 
\end{equation}
where $u^\varepsilon (t,\cdot)$ is the law of $Y^\varepsilon_t, t \ge 0$
and $Y^\varepsilon_0$ is distributed according to $u_0(x) dx$.
The sequence of laws of the  processes $(Y^\varepsilon)$ are tight,
but the   limiting process of a convergent subsequence  a priori
may not necessarily  solve  the SDE 
\begin{equation} \label{E1.2c}
Y_t = Y_0 + \int_0^t \Phi (u(s,Y_s)) dW_s. 
\end{equation}

However, this will be shown to be the case in the following two general situations.
\begin{enumerate}
\item The case when the initial condition $u_0$ is locally of bounded variation,
without any further  restriction on the coefficient $\beta$. 
\item The case when $\beta$ is strictly increasing after some zero,
 see Definition \ref{D4.11}, and
without any further restriction on the initial condition.
\end{enumerate}
\end{enumerate}
In this paper, we proceed as follows. Section 2
is devoted to preliminaries and notations.
In Section 3, we analyze an elliptic non-linear equation
 with monotone coefficients
which constitutes the basis for the existence of a solution to (\ref{PME}).
We recall some basic properties and we establish some other which will be useful later.
In Section 4, we recall the notion of $C^0$- solution to \eqref{PME} coming from an implicite scheme
of non-linear elliptic equations presented in Section 3. Moreover, we prove three   significant properties.
The first is that $\beta(u(t,\cdot))$ is in $H^1$, therefore continuous, for almost all $t \in [0,T]$.
The second is that the solution $u(t,\cdot)$ is locally of bounded variation if $u_0$ is.
The third is that if $\beta$  is strictly increasing after some zero, then
$\Phi(u(t,\cdot))$ is continuous for almost all $t$.
Section 5 is devoted to the study of the probabilistic representation of \eqref{PME}. \\
 Finally, we would like to mention that, in order to keep this paper
self-contained and make it accessible to a larger audience, we include 
the analytic background material and necessary (through standard) definitions.
Likewise, we tried to explain all details on the analytic delicate and 
quite technical parts of the paper which form the back  bone
of the proofs for our main result.


\section{Preliminaries}

We start with some basic analytical framework.

If $f: \R \rightarrow \R$ is a bounded function we will set
$\Vert f \Vert_\infty  = \sup_{x \in \R} \vert f(x) \vert. $
  By $C_b(\R)$ we denote the space of bounded continuous real functions and
by $C_\infty(\R)$  the space of all continuous 
functions on $\R$ vanishing at infinity.
 $\shd\left(\mathbb{R}\right)  $ will be the space of all  
 infinitely  differentiable functions  with compact support $\varphi:\mathbb{R}\rightarrow
\mathbb{R}$, and  $\shd^{\prime}\left(  \mathbb{R}\right)  $  will be 
its dual (the
space of Schwartz distributions).
   $\shs\left(  \mathbb{R}\right)  $  is the space of all rapidly decreasing infinitely
differentiable functions $\varphi:\mathbb{R}\rightarrow
\mathbb{R}$, and  $\shs^{\prime}\left(  \mathbb{R}\right)  $ will be
 its dual (the
space of tempered distributions). 

If $p \ge 1$  by $L^p(\R)$ (resp. $L^p_{\rm loc}(\R)$), we denote  
the space of all real Borel functions $f$ 
such that $\vert f \vert^p$   is integrable
(resp. integrable on each compact interval). We denote 
the space of all Borel essentialy bounded real functions  by $L^\infty(\R)$.
In several situations we will even omit $\R$.

We will use the classical notation $W^{s,p}(\R)$ for 
 Sobolev spaces,
 see e.g. \cite{adams}. $\Vert \cdot \Vert_{s,p}$ denotes the
 corresponding norm.
We will use the notation $H^s(\R)$ instead of
 $ W^{s,2}(\R)$. If $ s \ge 1$, this space is a subspace
of the space $C(\R)$ of real continuous functions.
We recall that, by Sobolev embedding,  $ W^{1,1}(\R) \subset
C_\infty(\R)$ and that each  $u \in  W^{1,1}(\R)$ has an absolutely
continuous version.
Let $\delta > 0$. We will denote by $< \cdot, \cdot>_{-1,\delta}$ the inner
 product
$$ <u,v>_{-1,\delta} = <(\delta - \frac{1}{2} \Delta)^{-1/2} u, 
(\delta - \frac{1}{2}  \Delta)^{-1/2} v>_{L^2(\R)},$$
 and by $\Vert \cdot \Vert_{-1,\delta}$ the corresponding norm.
For details about $(\delta - \frac{1}{2} 
\Delta)^{-s} $, see \cite{stein, Triebel} and
 also \cite{BRR},
section 2. 
In particular, given $s \in \R$, $ (\delta -  \frac{1}{2}  \Delta)^{s}$
maps $\shs'(\R)$ (resp. $ \shs(\R)$) onto itself.
If $u \in L^2(\R)$.
$$ (\delta -  \frac{1}{2}  \Delta)^{-1} u(x) = \int_\R K_\delta(x-y) v(y) dy,$$
with 
\begin{equation} \label{EKernel}
K_\delta \left(  x\right)  =
  \frac{1}{ \sqrt
  {2 \delta}} e^{- \sqrt {2\delta} \vert x \vert}.
\end{equation}
Moreover
 the map $(\delta -  \frac{1}{2}  \Delta)^{-1}$
 continuously maps $H^{-1}$ onto $H^{1}$
and a tempered distribution $u$ belongs to  $H^{-1}$ if and only if 
$ (\delta - \frac{1}{2}  \Delta)^{-1/2} u \in L^2. $
\begin{rem} \label{Rdelta}
$ L^1 \subset H^{-1}$ continuously.
Moreover for $u \in L^1$,
$$\Vert u \Vert_{-1,\delta} \le \Vert K_\delta \Vert_\infty ^{\frac{1}{2}}
\Vert u  \Vert_{L^1} = (2 \delta)^ {-\frac{1}{4}} \Vert u  \Vert_{L^1}.$$
\end{rem}

Let $T > 0$  be fixed.
For functions $(t,x) \rightarrow u(t,x)$,
the notation $u'$ (resp. $u''$) will denote the first (resp. second) 
derivative with respect to $x$.

Let $E$  be a  Banach space. 
One of the most basic notions of this paper 
is the one of a multivalued function (graph).
A {\bf multivalued function} (graph) $\beta$ on $E$ will be a subset
 of $E \times E$.
It can be seen, either as a family of couples $(e,f), e, f \in E$ and we will
write $f \in \beta(e)$ or as a function $\beta: E \rightarrow \shp(E)$.

We start with the definition in the case $E = \R$.

\begin{defi} \label{D2.4b}
A  multivalued function  $\beta$  defined on $\R$ 
with values in subsets of $\R$ is 
said to be {\bf monotone} if given $ x_1, x_2 \in \R$,
$(x_1 - x_2) (\beta(x_1) - \beta(x_2)) \ge 0$.

We say that $\beta$ is {\bf maximal monotone} 
(or a {\bf maximal monotone graph}) 
if it is monotone and if for one (hence all) $\lambda > 0$,
$\beta + \lambda id$ is surjective, i.e.
$$ {\cal R} (\beta + \lambda id) 
:=  \bigcup_{x \in \R} (\beta (x) + \lambda x) = \R.$$
\end{defi}

For a maximal monotone graph $\beta: \R \rightarrow 2^\R$, we define
a function  $j: \R \rightarrow \R$ by
\begin{equation} \label{E2.6prime}
 j(u) = \int_0^u \beta^\circ (y) dy, \ u \in \R,
\end{equation} 
where $\beta^\circ$ is the minimal section of $\beta$.
It fullfills the property that $\partial j = \beta$ in the sense of 
convex analysis see e.g. \cite{Barbu2}. In other words
$\beta$ is the subdifferential of $j$.  $j$ is convex, continuous and 
if $ 0 \in \beta(0)$, then $j \ge 0$.

We recall that one  motivation of this paper is the case
  where $\beta(u) = H(u - e_c) u$.
It can be considered as a multivalued map {\it by filling the gap}.
More generally, let us consider a monotone function $\psi$.
Then all the discontinuities are of jump type.
At every discontinuity point $x$  of $\psi$,
it is possible to {\it complete} $\psi$ 
by setting $\psi(x) = [\psi(x-), \psi(x+)]$.
Since $\psi$ is a monotone function,
the corresponding multivalued   function will be, of course, also
monotone.

Now we come back to the case of our general Banach space $E$ with norm
$\Vert \cdot \Vert.$
An operator $T: E \rightarrow E$ is said to be a {\bf contraction}
 if it is Lipschitz of norm less or equal to 1 and 
 $T(0) =  0$.

\begin{defi}
 A map $ A: E \rightarrow E$, 
or more generally
a multivalued map $A: E \rightarrow  \shp(E)$
is said to be  {\bf accretive}  if
for any $f_1, f_2, g_1, g_2  \in E$ 
such that $g_i \in A f_i, i = 1,2$,
 we have 
$$ \Vert f_1 - f_2 \Vert  \le  \Vert f_1 - f_2 + \lambda (g_1 - g_2)
 \Vert, $$
for any $\lambda  >    0$.
\end{defi}
This is equivalent to saying the following:  for any $\lambda   >    0 $,
 $(I + \lambda A)^{-1}$ is 
a contraction  on  $Rg(I + \lambda A)$.
We remark that a contraction is necessarily single-valued.

\begin{prop} \label{R27a}
Suppose that $E$ is a Hilbert space equipped with the scalar product
$(\;,\, )_H$.
 Then  $A$ is accretive if and only if $A$ is {\bf monotone} i.e.
$ ( f_1 - f_2,  g_1 - g_2)_H \ge 0$ for any $f_1, f_2, g_1, g_2  \in E$ 
such that $g_i \in A f_i, i = 1,2$, see  Corollary 1.3 of
\cite{sho97}.
 \end{prop}

\begin{defi} \label{D27a}
 An accretive  map $ A: E \rightarrow E$  (possibly multivalued)
is said to be  {\bf m-accretive}
if for some $\lambda  >  0 $, 
$I + \lambda A$ is surjective (as a graph in $E \times E$).

\end{defi}
\begin{rem} \label{Rsho}
An accretive map $ A: E \rightarrow E$ is m-accretive if
and only if $I + \lambda A$ is surjective for any $\lambda  > 0$.

So, $A$ is m-accretive, if and  only if 
for all $\lambda$ strictly positive, 
$(I + \lambda A)^{-1}$ is a contraction on $E$.

If $E$ is a Hilbert space, by the celebrated Minty's theorem, see e.g.
\cite{Barbu1}, a mapping $A: E \rightarrow E$  is m-accretive if it is
maximal monotone, i.e. it is monotone and has no proper monotone extension.
\end{rem}

Now, let us consider the case $E = L^1(\R)$,  so $E^* = L^\infty(\R)$.
The following  is  taken from \cite{BeC81}, Section 1.

\begin{theo} \label{pr1}
 Let $\beta: \R \rightarrow \R$ be a monotone (possibly multi-valued)
function such that  the corresponding graph is maximal monotone.
Suppose that $0 \in \beta(0) $.
Let  $  f \in E = L^1(\R)$.

\begin{enumerate}
\item
There is a unique $u \in L^1(\R)$ for which there is 
$w \in L^1_{\rm loc}(\R)$ such that 
\begin{equation} \label{pe1} 
u - \Delta w = f \quad {\rm in} \quad 
\shd'(\R), \quad w(x) \in \beta(u(x)), \quad  {\rm for \ a.e. } \quad
x \in \R,
\end{equation}
see Proposition 2 of  \cite{BeC81}.
\item Then, a (possibly  multivalued) operator
$A:=  A_\beta: D(A) \subset  E \rightarrow E$ is defined with
$ D(A) $  being the set  of $ u \in L^1(\R) $ for which 
 there is $w \in L^1_{\rm loc}(\R)$
such that  $w(x) \in \beta(u(x)) $ for \ a.e. \ $x \in \R$ and $\Delta w 
\in L^1 (\R)$ and 
for $u \in  D(A)$
$$ A u = \{ - \frac{1}{2} \Delta w \vert w \ 
{\rm  as \ in \ definition \ of \ }
D(A)\}.$$ 
This is a consequence of the remarks following Theorem 1 in   \cite{BeC81}.

In particular, if $\beta$ is single-valued, then $A u = - 
\frac{1}{2} \Delta \beta(u)$.
(We will adopt this notation also if $\beta$ is multi-valued).

\item The operator  $A$ defined in 2. above is m-accretive
on $E = L^1(\R)$, see Proposition 2 of  \cite{BeC81}.
Moreover $\overline {\shd(A)} = E$. 

\item
We set  $J_\lambda = (I + \lambda A)^{-1}$, which is a
  single-valued operator.
If $f \in L^\infty (\R) $, then $ \Vert \vert J_\lambda f \Vert_\infty
  \le \Vert f \Vert_\infty$, 
see Proposition 2 (iii) of  \cite{BeC81}.
In particular, for every  positive integer  $n$,
$ \Vert J_\lambda^n f \Vert_\infty  \le \Vert f \Vert_\infty$.
\end{enumerate}
\end{theo}

Let us summarize some important results of
 the theory of non-linear semigroups, see for
instance \cite{E77,  Barbu1, Barbu2, BeBrC75} or the
 more recent monograph \cite{sho97},
which we shall use below.
Let $A: E \rightarrow E$ be a (possibly multivalued) accretive operator.
We consider the equation 
\begin{equation} \label{ENonLinEv}
0 \in u'(t) + A (u(t)),  \quad  0 \le  t \le T.
\end{equation}
A function $u: [0,T]  \rightarrow E$ which is absolutely continuous
such that for a.e. $t$, $u(t,\cdot) \in D(A)$ and 
fulfills (\ref{ENonLinEv}) in the following sense is called
{\bf strong solution}.

There exists $\eta: [0,T] \rightarrow E$, Bochner integrable, such
that $\eta(t) \in A(u(t))$ for a.e. $t \in [0,T]$ and 
$$ u(t) = u_0 - \int_0^t \eta (s) ds, \quad  0  <   t \le T. $$

A weaker notion for (\ref{ENonLinEv}) is the so-called 
{\bf $C^0$- solution}, see chapter IV.8 of \cite{sho97},
or {\bf mild solution}, see \cite{Barbu2}.
In order to introduce it, 
 one first defines the notion of $\varepsilon$-solution related
to (\ref{ENonLinEv}).

An {\bf $\varepsilon$-solution} is a discretization 
$$ \shd = \{0= t_0 < t_1 < \ldots < t_N = T \} $$
and an $E$-valued step function 
$$ u^\varepsilon (t) = \left \{
 \begin{array}  {ccc}
 u_0 &: & t= t_0 \\
u_j \in D(A)  &: & t \in ]t_{j-1}, t_j ],
\end{array}
\right.
$$ for which
$ t_j - t_{j-1} \le \varepsilon $ for 
$1 \le j \le N$,
and 
$$ 0 \in \frac{u_j - u_{j-1}}{t_j - t_{j-1}} + A u_j, 1 \le j \le N. $$
We remark that, since $A$ is maximal monotone, $u^\varepsilon$ is determined by
$\shd$ and $u_0$, see Theorem \ref{pr1} 3.
\begin{defi} \label{DC0Sol}
A {\bf $C^0$- solution} of (\ref{ENonLinEv}) is an $u \in C([0,T]; E)$
such that for every $\varepsilon >  0$, there is an
$\varepsilon$-solution  $u^\varepsilon$ of (\ref{ENonLinEv}) with
$$ \Vert u(t) - u^\varepsilon (t) \Vert \le \varepsilon, \quad 0 \le t \le T.$$
 \end{defi}

\begin{prop} \label{Panls}
Let $A$ be a maximal monotone (multivalued) operator  on a Banach
space $E$.
We set again $J_\lambda : = (I + \lambda A)^{-1}, \lambda > 0  $.
 Suppose  $u_0 \in \overline{D(A)} $. Then: 
\begin{enumerate}
\item 
 There is a unique  $C^0$- solution $u: [0,T]
\rightarrow E$  of   (\ref{ENonLinEv})
\item  
 $u(t) = \lim_{n \rightarrow \infty} J_{\frac{t}{n}}^n u_0$
uniformly in $t \in[0,T]$.

\end{enumerate}
\end{prop}

\begin{proof} \

1) is stated in
Corollary IV.8.4. of \cite{sho97} and
 2)  is contained in Theorem IV 8.2 of \cite{sho97}.
\end{proof}

The complications coming from the definition of $C^0$-solution arise
because the dual  $E^*$ of $E= L^1(\R)$ is not uniformly convex.
In general a  $C^0$-solution is not absolutely continuous 
and not a.e. differentiable, so
it is not a strong solution.
For uniformly convex Banach spaces, the situation is much easier.
Indeed, according
to  Theorem IV 7.1 of \cite{sho97},
for a given $u_0 \in D(A)$, there  would exist a (strong)
solution $u: [0,T] \rightarrow E$ to (\ref{ENonLinEv}).
Moreover,   Theorem 1.2 of
\cite{CE75} says the following.
Given  $u_0 \in \overline{D(A)}$ and given a sequence $(u_0^n)$  in  $
  D(A)$ converging to $u_0$, 
then the sequence of the corresponding strong solutions  $(u_n)$  
would converge to the unique
 $C^0$-solution of the same equation.

\section{Elliptic equations with monotone coefficients}

\setcounter{equation}{0}

Let us fix our assumptions on $\beta$ which we assume to be in force in
this entire section.
\begin{hypo} \label{H3.0}
Let $\beta: \R \rightarrow 2^\R$ be a maximal monotone graph
with the property  that  there exists  $c > 0$  such that
\begin{equation} \label{E3.0}
 w \in \beta(u) \Rightarrow \vert w \vert \le c \vert u \vert. 
\end{equation}
\end{hypo}
We note that \eqref{E3.0} implies that $\beta(0) = 0$, hence
$j(u) \ge 0$, for any $u \in \R$, where $j$ is defined
in \eqref{E2.6prime}. Furthermore, by Hypothesis \ref{H3.0},
\begin{equation} \label{E3.1prime}
 j(u) \le \int_0^{\vert u \vert } \vert \beta^\circ (y) \vert dy \le c
 \vert u \vert^2.    
\end{equation}

We recall from \cite{BeC81} that the first ingredient 
to study well-posedness of equation \eqref{PME}  is the following
elliptic equation
\begin{equation}\label{EE}
   u- \lambda \Delta \beta (u ) \ni f
\end{equation} 
 where $f\in L^1 (\R)$ and $u$ is the unknown function in $L^1 (\R)$.
 
 \begin{defi}\label{D31} Let  $f\in L^1 (\R)$. Then  $u\in L^1 (\R)$
  is called a solution  of \eqref{EE} 
if there  is  $w\in L^1_{\text{loc}}$ with
 $w \in \beta (u)$ a.e. and
  \begin{equation}\label{EE1}
   u -  \lambda \Delta  w = f  
  \end{equation} 
 in the sense of distributions. 
\end{defi}
 According to Theorem 4.1 of \cite{BeBrC75},
 and Theorem 1, Ch.1, of \cite{BeC81},
  equation  \eqref{EE} 
admits a unique solution. Moreover, $w$ is also uniquely determined by $u$.
 Sometimes, we will also call the couple $(u,w)$  the solution to \eqref{EE1}.

 We recall some basic properties of the couple $(u,w)$.
 \begin{lemma}\label{L31}
   Let $(u,w)$ be the unique solution of \eqref{EE}.
   Let $J^{\lambda}_ {\beta }: L^1 (\R) \rightarrow L ^1 (\R)$ be the map which
   associates the solution $u$ of \eqref{EE} 
to  $f \in L^1(\R)$.   We have the following:
   \begin{enumerate}
      \item $  J^{\lambda}_ {\beta } 0  = 0$.
\item $  J^{\lambda}_ {\beta }   $ is a contraction in the sense that
      \[\norm{ J^{\lambda}_ {\beta }    (f _1 ) -  J^{\lambda}_ {\beta } (f_2)}_{L^1}\leq \norm{f _1 - f_2}_{L^1}\]
      for every $f_1, f_2\in L^1$.
      \item If $f \in L^1  \bigcap  L^\infty $, 
then $\norm u _\infty \leq \norm f_\infty$
      \item If $f\in L^1 \bigcap L^2$, then $u \in L^2 $ and
      \[\int _\R u ^2 (x) dx \leq \int _\R f ^2 (x) dx.\]
and $  \int _\R j (u) (x) dx \leq \int _\R j (f) (x) dx 
\le {\rm const} \Vert f \Vert_{L^2} $.
\item  Let  $f \in L^1$. Then 
$ w, w' \in W^{1,1} \subset C_\infty(\R).$
Hence, in particular $w \in W^{1,p}$ for any $ p \in [0, \infty]$.

\end{enumerate}
 \end{lemma}
 \begin{proof}\ 
 \begin{enumerate}
\item is obvious and comes from uniqueness of \eqref{EE}.
  \item See Proposition 2.i) of \cite{BeC81}.
\item See Proposition 2.iii) of \cite{BeC81}.
\item This follows from \cite{BeBrC75}, Point III, Ch. 1
and \eqref{E3.1prime}.
\item We define $ g := \frac{1}{2 \lambda} (w + f - u)$. Since  $f\in L^1$,
also  $u \in L^1$, hence  $w\in L^1$ by \eqref{E3.0}.
Altogether it follows that  $g\in L^1$.  \eqref{EE1} 
and  \eqref{EKernel} imply that $ w = K_\delta \star g$
with $\delta =  \frac{1}{2 \lambda}$ 
and hence
$$ w' =  K'_\delta \star g =  \int {\rm sign}(y-x)
e^{- \lambda^{-\frac{1}{2}}  \vert x-y \vert} g(y) dy, $$
where
$$
{\sign} \ x = \left \{
\begin{array}{ccc}
-1 &:& x < 0 \\
0 &:& x = 0 \\
1 &: & x > 0.
\end{array}
\right. 
$$
This implies $ w, w' \in L^1 \bigcap L^\infty$. By \eqref{EE1}
we  know that also $w'' \in L^1$, hence $w, w'
\in  W^{1,1} (\subset C_\infty )$.
          \end{enumerate}
      \end{proof}
\begin{rem}\label{R31bis}
Let $\delta > 0$. The same results included in Lemma \ref{L31}
are valid for the equation 
\begin{equation} \label{EEdelta}
 u + \lambda \delta \beta(u) - \lambda \Delta (\beta(u)) \ni f.
\end{equation}
In fact,  \cite{BeBrC75} treats the  equation
$\Delta v + \gamma(v) \ni f$,
with $\gamma: \R \rightarrow 2^\R$ a maximal monotone graph.
 We reduce equation
 \eqref{EE} and \eqref{EEdelta} to 
this equation, by setting  $v = \lambda \beta(u), \gamma(v) =
 - \beta^{-1}(\frac{v}{\lambda})$,
where $\beta^{-1}$ is the inverse graph of $\beta$,
 and  
setting  $v = \lambda \beta(u), \quad \gamma(v) =  
- \beta^{-1}(\frac{v}{\lambda})  - \delta v$, respectively.
In both cases $\gamma$ is a maximal monotone graph.
\end{rem}

Since 
$w'' \in L^1$,
 \eqref{EE1} can be written as
\begin{equation}\label{EE3}
   \int_\R u (x) \varphi (x) dx - \lambda \int _\R w''(x) \varphi (x) dx= \int _\R f(x) \varphi (x) dx \quad
   \forall \varphi \in L^\infty (\R).
\end{equation} 

Since $w\in L^\infty$, we may replace $\varphi$ by $w$ in \eqref{EE3}.
In addition, $w' \in L^2$, so by a simple approximation argument, 
it follows that
\begin{equation}\label{EE4}
   \int_\R (u(x) - f(x) ) w(x) dx + \lambda \int _\R {w'}^2 (x) dx =0.
\end{equation}
Now, we are ready to prove the following.
        
\begin{lemma}\label{L.3.2}
  Let $f \in L^1 \bigcap L^2  $ and $(u,w)$ be a solution to \eqref{EE}. Then\\
 $\int _\R (j (u) - j (f) ) (x) dx \leq -\lambda  \int _\R {w'}^2 (x) dx$.
\end{lemma}
\begin{proof} \ 
 By definition of the subdifferential and since
 $w(x) \in \beta (x)$ for a.e. $ x \in \R$, we have
   \begin{equation}\label{EE5}
      (j(u) - j(f)) (x) \leq w(x) (u-f) (x) \quad \text{a.e.} \  x \in \R.
   \end{equation}
   Again \eqref{EE4} implies the result after integrating  \eqref{EE5}.
\end{proof}
We go on analysing the local bounded variation character of the solution  $u$ of
\eqref{EE}.

If $f: \R \rightarrow \R$, for $h \in \R$, we define
\begin{equation} \label{E3.8prime}
f^h(x) = f (x+h) -f(x).
\end{equation}

Writing ${w ^h} '':= (w ^h) ''$
we observe that
\begin{equation} \label{(3.9)prime}
u^h - \lambda {w ^h} '' = f ^h,
\end{equation}
where $w(x ) \in \beta (u(x)), $ and  $ w (x + h ) \in \beta ( u(x+ h))$ a.e.

Let $\zeta\geq 0$ be a smooth function with compact support.
\begin{lemma}\label{L33}
 Assume $\beta$ is strictly monotone, i.e.
\begin{equation} \label{(3.9)second} 
 \beta(x) \bigcap \beta(y) = \emptyset  \ {\rm if} \  x \neq y.
\end{equation}
Let $u$ be a solution of \eqref{EE}.
  Then,  for each $h \in \R$
   \begin{equation} \label{(3.10)}
      \int _\R \zeta(x) \vert u ^h (x)\vert  dx \leq
 \int _\R \zeta(x)
 f^h (x) \sign(w^h(x)) dx + c \Vert \zeta''' \Vert_\infty 
 \lambda | h | \norm u _{L^1}.
   \end{equation}
   where $c$ is the constant from \eqref{E3.0}.
\end{lemma}
\begin{proof} \ 
 \eqref{(3.9)prime} gives
 \begin{equation}\label{EE5bis}
   \int _\R u^ h (x ) \varphi (x) dx
 =  \int _\R  (\lambda {w ^{h}}''(x) +  f ^h (x)) \varphi (x) dx 
, \quad    \forall \varphi \in L^\infty (\R).
 \end{equation}   

 We set $\varphi(x) = \sign (u^h (x) ) \zeta (x)$.
By \eqref{(3.9)second} we have $w^h \neq 0$ on $\{ u^h \neq 0 \}$,
\ $dx $ a.e. Hence, by strict
 monotonicity  we have $\varphi(x) =  \sign (w^h(x)) \zeta(x) $ a.e.
on $\{ u^h \neq 0\}$.
By \eqref{(3.9)prime}, up to a Lebesgue null set, we have
$ \{u_h = 0 \} = \{\lambda {w^{h}}'' + f^h = 0 \}$. Hence 
\eqref{EE5bis} implies
\begin{eqnarray*}
 \int _\R \zeta (x) | u^h(x) | dx &=&
 \int _{\{u^h \neq 0 \}}  (\lambda {w ^{h}}''(x) +  f ^h (x)) 
\sign(w^h(x))\zeta (x) dx  \\
&=&  \lambda \int _\R {w^{h}}'' (x)
 \sign (w^h)(x) \zeta (x) dx\\
 &+&  \int _\R f^h (x) \sign (w^h(x)) \zeta (x) dx.
\end{eqnarray*}
It remains to control
\begin{equation}\label{EE6}
  \lambda  \int _\R {w^{h}}'' (x)  \sign(w^h(x))  \zeta (x) dx.
\end{equation}


Let $\varrho = \varrho_L: \R \rightarrow \R$, be an odd smooth function such
that $\varrho \leq 1$ and 
$\varrho (x) = 1$ on $[\frac{1}{L}, \infty[$.
\eqref{EE6} is the limit when  $L$ goes to infinity of
\begin{align*}
  \lambda \int _\R {w^{h}}  '' (x) \varrho( w ^h (x)) \zeta (x) dx 
 &=&-\lambda \int _\R {w^{h}}  ' (x)^2 \varrho'( w ^h (x)) \zeta (x) dx \\
 &-& \lambda \int _\R {w^{h}}  ' (x) \varrho( w ^h (x)) \zeta '(x) dx.
\end{align*}
Since the first integral of the right-hand side of  the
previous expression is positive,
 \eqref{EE6}
 is upper bounded by
the {\rm limsup} when $L $ goes to infinity of
 $$  -\lambda \int _\R {w^{h}}  ' (x) \varrho( w ^h (x)) \zeta' (x) dx
 =  - \lambda \int _\R (\tilde \varrho( w ^h (x)))' \zeta' (x) dx,
$$
where $\tilde \varrho (x) = \int _0 ^x \varrho (y ) dy$.
But the previous expression is equal to
$$
\lambda \int _\R \tilde \varrho (w^ h(x)) \zeta '' (x) dx
    = \lambda\int _R \tilde \varrho (w (x)) 
( \zeta'' (x - h ) -\zeta '' (x)) dx.
$$
Since $\tilde \varrho(x) \le \vert x \vert,$ pointwise and $ w \in
\beta (u)$ a.e., $u \in L^1$, the previous integral is bounded by
\begin{equation}
 \lambda \norm {\zeta'''}  _\infty | h | \int _\R | w (x) | dx    \leq
c \lambda  \norm {\zeta '''}_\infty |h|
\int _\R |u(x)| dx, 
\end{equation} 
   with $c$ coming from \eqref{E3.0}.
\end{proof}
\begin{rem} \label{RBV}
Using similar arguments as in Section 4 below, we can show that $u$ is
locally of bounded variation whenever $f$ is.
We have not emphasized this result since we will not directly use it.
\end{rem}

\section{Some properties of the porous media equation}

\setcounter{equation}{0}

Let   $\beta : \R \rightarrow 2^\R$.
 Throughout this section, we assume that
$\beta$ satisfies Hypothesis \ref{H3.0}.
Our first aim is to prove Theorem \ref{T4.5}
below, for which we need some preparations.
 Let $u_0\in (L^1 \cap L^\infty)(\R)$.
We recall some results stated  in \cite{BRR}
as Propositions 2.11 and 3.2.

\begin{prop}\label{R4.1}
   \begin{enumerate}
      \item Let $u_0 \in (L^1 \bigcap L^\infty)(\R)$.
Then, there is a unique solution to \eqref{PME} in the sense of distributions.
   This means that there exists a unique
 couple $(u, \eta_u) \in \left (L^1 \cap L^\infty \right)([0,T]\times
   \R)^2$
 with
\begin{equation} \label{E4.111}   
\int _\R u(t,x) \varphi (x) dx = \int _\R u_0 (x) \varphi(x) dx +
\frac 12 \int _0^t \int _\R\eta_u (s, x ) \varphi'' (x) dx \;ds, 
\end{equation}
where   $\varphi \in C_\circ^\infty (\R)$.
Furthermore, $t \rightarrow u(t,\cdot)$ is in
$C([0,T], L^1)$ and
$ \eta_u(t,x) \in \beta(u(t,x))$ for 
$dt \otimes dx$-a.e. $(t,x) \in [0,T] \times \R$. 
   \item We define the multivalued map 
$A = A_\beta : E \rightarrow E,$ $E = L^1 (\R)$, where
   $D(A)$ is the set of all
 $u\in L^1$ for which there is $w\in L^1 _{\rm loc}(\R)$ 
such that
   $w(x) \in \beta (u(x))$ a.e. $x\in \R$ and $\Delta w \in L^1 (\R)$.
 For $w \in D(A)$ we
   set
   \[A u = \{ -\frac{1}{2} \Delta w|
 w \text{ as in the definition of }D (A)\}.\]
   Then $A$ is $m$-accretive on $L^1(\R)$.
 Therefore there is a unique $C^0$-solution of the evolution problem
   \begin{align*}
      \begin{cases}
0 \in u'(t) + A u(t), \\
      u(0) = u_0.
      \end{cases}
   \end{align*}
   \item The $C^0$-solution under 2. 
 coincides with the solution in the sense of distributions under 1.
   \item $\norm{ u } _\infty \leq \norm{u_0}_\infty$.
   \item Let $\beta^\varepsilon (u) = \beta (u) + \varepsilon u , \quad \varepsilon > 0$ and
   consider the solution $u^{(\varepsilon)} $ to
   \[\begin{cases}
      \partial _t u^{(\varepsilon)} = \frac 12 \beta^\varepsilon
 (u^{(\varepsilon)} )''\\
      u^{(\varepsilon)} (0,\cdot) = u_0.
     \end{cases}
\]
Then $u^{(\varepsilon)} \rightarrow u$ in $C([0,T], L^1 (\R))$ when
 $\varepsilon \rightarrow 0$, see \cite{BeC81}.
\end{enumerate}
\end{prop}
\begin{cor} \label{R4.11} 
 We have $u(t, \cdot) \ge 0$ a.e.
 for any $t \in [0,T]$.
Moreover $\int_\R u(t,x) dx =  \int_\R u_0(x) dx = 1 $, for any $t \ge 0$.
\end{cor}
\begin{proof} \
In fact the functions $u^{(\varepsilon)}$ introduced in point 4.
of Proposition \ref{R4.1} have the desired property.
Taking the limit when $\varepsilon$ goes to zero,
the assertion follows.

\end{proof}

\begin{rem} \label{R4.211}
Uniqueness  to \eqref{E4.111} holds even only with the assumptions 
 $\beta$ monotone, continuous at zero and $\beta(0) = 0$, see \cite{BrC79}.
\end{rem}
Below we fix on an initial condition $u_0 \in L^1 \bigcap L^\infty$.'¡
\begin{lemma}\label{R4.2}
Let $\varepsilon > 0$.
We consider an
 $\varepsilon $-solution given by
\[\mathcal D= \{ 0 = t_0 < t_1 < \ldots  < t_N = T \}\]
and
\[u^\varepsilon (t) = \begin{cases}
               u_0, & \; t = 0\\ 
               u_j, &\; t \in ] t_{j-1}, t_j]
              \end{cases},
\]
  for which for $1 \leq j \leq N$

  \[\begin{cases}
               u_j- \frac{t_j - t_{j-1}}{2} w'' _j = u_{j-1},\\
               w _j \in \beta(u_j)\quad \text{a.e.}.
              \end{cases}
\]
We set
\[\eta^\varepsilon (t, \cdot) = \begin{cases}
               \beta(u_0)& \quad : t =0,\\
               w_j & \quad  : t \in ]t_{j-1}, t_j].
              \end{cases}
\]

When $\varepsilon \rightarrow 0$, $\eta^\varepsilon $ converges weakly
 in  $L^1([0,T] \times \R)$
 to
$\eta_u$, where $(u,\eta_u)$ solves equation \eqref{PME}.
Furthermore, for $p = 1$ or $p = \infty$,
\begin{eqnarray} \label{E4.1prime}
\sup_{t \le T} \Vert u^\varepsilon (t, \cdot) \Vert_{L^p} & \le& 
\Vert u_0 \Vert_{L^p} \nonumber \\
  {\rm and} &&\\  
\sup_{t \le T} \Vert \eta^\varepsilon (t, \cdot) \Vert_{L^p} & \le& 
c \Vert u_0 \Vert_{L^p},  \nonumber
\end{eqnarray}
where $c$ is as in Hypothesis \ref{H3.0}. Hence,
\begin{equation} \label{E4.1second}
\sup_{t \le T} \Vert u (t, \cdot) \Vert_{L^p}  \le 
 \Vert u_0 \Vert_{L^p}.
\end{equation}
and 
\begin{equation} \label{(E4.1second)}
 \Vert \eta \Vert_{L^r([0,T] \times \R)}  \le 
 c T^{\frac{1}{r}}  \Vert u_0 \Vert_{L^\infty}^{\frac{r-!}{r}}
 \Vert u_0 \Vert_{L^1}^r
\end{equation}
for all $r \in [1, \infty[$.
\end{lemma}

\begin{proof} \
See  point 3. in the proof of the Proposition 3.2 in \cite{BRR}.
\eqref{E4.1prime}
 follows by Lemma \ref{L31}, 1-3. and
Hypothesis \ref{H3.0} by induction. 
\eqref{E4.1second} is an immediate consequence of the first part
of \eqref{E4.1prime} and the fact that $u$ is a 
$C^0$ solution.
\eqref{(E4.1second)} follows by an elementary interpolation argument.
Indeed, for $r \in [1, \infty[, \quad r':= \frac{r}{r-1}, \ 
\varepsilon_n \rightarrow 0$ and
$\varphi \in (L^r \bigcap L^\infty)([0,T] \bigcap \R)$
we have
\begin{eqnarray*}
\left \vert \int_0^T \int_\R \varphi(t,x) \eta_u(t,x) dx dt \right \vert &=&
\lim_{n \rightarrow \infty} 
\left \vert \int_0^T \int_\R \varphi(t,x) \eta^{\varepsilon_n}(t,x) dx dt 
\right \vert \\
&\le & \Vert \varphi \Vert_{L^{r'}([0,T] \times \R)}
  \liminf_{n \rightarrow \infty} 
\left( \int_0^T \int_\R \vert \eta^{\varepsilon_n}(t,x) \vert^r dx dt \right)
^{\frac {1} {r}} \\
& \le &  \Vert \varphi \Vert_{L^{r'}([0,T] \times \R)} 
 \liminf_{n \rightarrow \infty} 
\sup_{t \le T} \Vert \eta^{\varepsilon_n}(t,\cdot) \Vert^{\frac{r-1}{r}}_
{L^\infty} T^{\frac{1}{r}} 
\sup_{t \le T} \Vert \eta^{\varepsilon_n}(t,\cdot) \Vert_
{L^1}^{\frac{1}{r}} \\
& \le &  \Vert \varphi \Vert_{L^{r'}([0,T] \times \R)}
T^{\frac{1}{r}} c  \Vert u_0 \Vert^{\frac{r-1}{r}}_
{L^\infty}  \Vert u \Vert^{\frac{1}{r}}_
{L^1},
\end{eqnarray*}
where we used the second part of \eqref{E4.1prime}
in the last step.
\end{proof}

If not mentioned otherwise, in the sequel for
$ N > 0$ and $\varepsilon = \frac{T}{N}$, we will consider the
subdivision 
\begin{equation} \label{E4.1third}
\mathcal D = \{ t_i = \varepsilon i, \; 0\leq i\leq N\}.
\end{equation}

We  now  discuss some properties of the solution exploiting
the fact that the initial condition is square integrable.

\begin{prop}\label{P4.3}
   Let $u_0 \in (L^1 \bigcap L^\infty) (\R)$.
 Then the solution $(u, \eta _u) 
\in  (L^1 \cap L^\infty)([0,T] \times \R)^2$
of   \eqref{PME}  has the following properties.
   \begin{enumerate}[a)]
      \item $\eta_u(t,\cdot)$ is absolutely continuous for a.e.
 $t \in [0,T]$ and
      $\eta_u \in L^2 ([0,T]; H^1(\R))$.
      \item
      \[\int _\R j (u(t,x)) dx + \frac 12 \int _r^t \int _\R (\eta'_u)^2
 (s,x) dx ds \leq
      \int _\R j(u(r,x)) dx \quad \forall \;0 \leq r \le t \le T.\]
 In particular $t \mapsto \int _\R j (u( t,x)) dx$ is  decreasing and
\begin{equation} \label{E4.0}  
\int _{[0,T)}\int _\R (\eta_{u}')^2(s,x)  dx ds \leq 2 \int _\R j
  (u_0 (x)) dx := C \le \left( {\rm const.} \Vert u_0 \Vert^2_{L^2}
 < \infty \right). 
\end{equation}
\item $t \mapsto \int j(u(t,x)) dx$ is continuous on $[0,T]$.
\end{enumerate} 
\end{prop}

\begin{proof} \
   We consider the scheme considered in Lemma \ref{R4.2}
 corresponding to $\varepsilon = \frac{T}{N}$.
   By Lemma \ref{L31} 5., we have $w_j \in H^1(\R)$, $1\leq j \leq N$, and
by Lemma \ref{L.3.2}  
 \begin{equation}\label{4.1}
      \int_\R j (u_i (x)) dx + \frac \varepsilon 2 \int _\R w _i' (x)^2 dx
      \leq \int _\R j (u_{i-1}(x)) dx \quad \forall i = 1, \ldots , N.
   \end{equation}
  Hence for any $0 \leq l \leq m \leq N$
  \begin{equation}\label{4.2}
   \int _\R j (u_m (x)) dx + \frac{\varepsilon}{2} 
 \sum_{i=l+1}^m \int _\R (w_i')(x)^2 dx \leq
   \int _\R j(u_l (x)) dx.
  \end{equation}
Using the notation introduced in Lemma \ref{R4.2}, for all
$ 0 \le r \le t \le T$, we obtain
\begin{equation}\label{4.7}
 \int _\R j (u^\varepsilon (t,x)) dx + \frac 12 \int _{r }^{t}\int _\R
 {(\eta^\varepsilon} ')^2   (s,x) ds dx
 \leq \int _\R j (u^\varepsilon (r,x)) dx.
\end{equation} 
  On the other hand,  Lemma \ref{R4.2} and Lemma \ref{L31} 4. 
imply that
\begin{equation} \label{4.5prime}  
\norm{u^\varepsilon (t) } _{L^2}
 \leq \norm{u_0}_{L^2}\quad \forall t \in [0,T].
\end{equation}
Therefore,
\begin{equation}\label{4.8}
 \int _0 ^T ds \int _\R u^\varepsilon (s,x)^2 dx \leq T \norm{u_0}^2_{L^2}.
\end{equation} 

Since $|\beta(u)| \leq c |u|$, \eqref{4.8} implies that
\begin{equation}\label{4.9}
   \sup_{\varepsilon > 0}\int _0 ^T ds
 \int _\R{ \eta^{\varepsilon}}  (s, x)^2 dx < \infty.
\end{equation}
\eqref{4.9} and \eqref{4.7} say that
$\eta^\varepsilon, \varepsilon > 0,$ 
are bounded in $L^2 ([0,T]; H^1(\R))$. There is then a
subsequence $(\varepsilon_n)$ with $\eta^{\varepsilon_n}$
 converging weakly in $L^2 ([0,T]; H^1(\R))$
and therefore also weakly in $L^2 ([0,T]\times \R)$ to some $\xi$.
According to Lemma \ref{R4.2} and the uniqueness of the limit, 
it follows $\eta_u = \xi$ and
so $\eta_u \in L^2 ([0,T]; H^1(\R))$, which implies a).
We recall that
\begin{equation}\label{4.8prime}
\int _0^T ds \int_\R (\eta'_u)^2 (s,x) dx \leq
\liminf _{\varepsilon\rightarrow 0}\int _0^T ds \int _\R 
({\eta ^\varepsilon }')^2 (s,x) dx.
\end{equation}
In fact the sequence $({\eta^\varepsilon}')$ is weakly relatively 
compact in $L^2 ([0,T] \times \R)$.
It follows by \eqref{E3.1prime} and \eqref{4.5prime} that
$j(u^\varepsilon(t)), \varepsilon > 0$, are uniformly integrable for each 
$t \in [0,T]$. 
Since $j$ is continuous, and $u^\varepsilon (t, \cdot)
 \rightarrow u(t, \cdot)$ 
in $L^1 (\R)$ for each 
$t \in [0,T]$, it follows that $j(u^\varepsilon (t,\cdot))
 \rightarrow j(u(t, \cdot))$ as $\varepsilon  \rightarrow 0$
in $L^1 (\R)$ for each 
$t \in [0,T]$.
\eqref{4.9} and \eqref{4.7} imply that
\begin{equation}\label{4.10}
   \int _\R j (u(t,x)) dx + \frac{1}{2}
 \int _r ^t 
\int _\R (\eta' _u)^2 (s,x) dx ds
   \leq \int _\R j(u(r,x)) dx,
\end{equation}
for every $0\leq r \leq t \leq T$, which is inequality b). \\
To prove c),  by \eqref{E4.1second} and Lebesgue dominated convergence theorem,
using again that $u_0 \in (L^1 \bigcap L^\infty)(\R) \subset L^2,$  
we deduce that $ t \mapsto u(t, \cdot)$ is in $C([0,T], L^2)$
since it is in  $C([0,T], L^1)$ by Proposition \ref{R4.1}, 1.
Now let $t_n \rightarrow t$ in $[0,T]$ as $n \rightarrow \infty$, then
$u(t_n, \cdot) \rightarrow u(t, \cdot)$ in $L^2$ as $n \rightarrow \infty$,
in particular $\{u^2(t_n, \cdot) \vert n \in \N \}$ 
is  equiintegrable,
hence by \eqref{E3.1prime}
 $\{j(u(t_n, \cdot)) \vert n \in \N \}$
is equiintegrable.
 Since
$j$ is continuous, assertion c) follows. 
\end{proof}

\begin{cor}\label{R4.4}
 \begin{equation}\label{E4.8}
  \int _{\R } u^2 (t,x) dx \leq  \norm{u_0}^2_{L^2}, \forall t \in [0,T].
 \end{equation} 
\end{cor}
\begin{proof} \
The result follows by Fatou's lemma, from \eqref{4.5prime}.
\end{proof}

Inequality \eqref{4.10} will be shown in Theorem \ref{T4.5} below 
  to be indeed an equality.

\begin{rem}\label{R00}
According to Proposition \ref{R4.1} 1., for every 
$\varphi \in C_0^\infty(\R)$, we have
 \begin{equation}\label{4.11}
   \int _\R u (t,x) \varphi(x) dx - \int _\R u_0(x) 
\varphi(x) dx = \frac 12 \int _0^t 
 \int _\R \eta _u(s, x)\varphi '' (x) dx ds.
 \end{equation}
 Since $s\mapsto \eta_u(s, \cdot)$ belongs to $L^2 ([0,T]; H^{1}(\R))$,
by Proposition \ref{P4.3} a),
 we have
 $s \mapsto \eta_u''(s, \cdot) \in L^2 ([0,T]; H^{-1}(\R))$.
 This, together with \eqref{4.11},
imply that $t \mapsto u(t,\cdot)$ is absolutely continuous from
 $[0,T]\rightarrow H^{-1}(\R)$. So, in $H^{-1}(\R)$ we have
 \begin{equation}\label{4.11bis}
   \frac{d}{dt} u (t, \cdot)= \frac 12 \eta_u'' (t, \cdot) 
\quad t \in [0,T]\quad \text{a.e.}.
 \end{equation}
\end{rem}
Before proving that \eqref{4.10} is in fact
 an equality,  we need to improve the upper bound established in \eqref{E4.0}.

\begin{prop}\label{L4.3a} In addition to Hypothesis \ref{H3.0},
we suppose  that
\begin{equation} \label{ASS}
\beta(\R) =  \R.
\end{equation}
 Then there is a constant $C > 0$ such that
\begin{equation} \label{E4.12prime} 
\int_\R \eta _u ' (t,x) ^2 dx \leq C \quad {\rm 
for \ a.e.}   t \in [0,T].
\end{equation}
\end{prop}
This proposition will be important to prove that the real function
$ t \mapsto \int_\R j(u(t,x)) dx$ is absolutely continuous.

\begin{proof}\
 We equip $H = H^{-1}(\R)$ with the inner product
 $\bracket{\cdot, \cdot} _{-1, \delta}$
 where $\delta \in ]0,1]$ and
 \[\bracket{u,v}_{-1, \delta} = \bracket{ (\delta - 
\frac{1}{2} \Delta )^{-\frac 12} u,(\delta - \frac 12 \Delta)
^{-\frac 12} v}_{L^2 (\R)},\]
and corresponding norm $ \Vert \cdot  \Vert_{-1, \delta} $.
 We define $\Gamma : H \rightarrow [0,\infty]$ by
 \[\Gamma (u ) =
 \begin{cases}
  \int _\R   j(u(x)) dx, \quad &\text{if } u \in L^1 _{\text {loc}} \\
  + \infty, & \text{otherwise}.
 \end{cases}
\]
and $\mathcal D(\Gamma) = \{ u \in H| \; \Gamma (u) < \infty\}$.
We also consider
\[D (A_\delta) = \{ u \in \mathcal D (\Gamma) | \;\exists\,
 \eta _u \in H^1 , \eta _u \in \beta (u) \text{ a.e.}\}.\]
For  $u \in  D(A_\delta),$
we set $A_\delta u = \{ (\delta - \frac 12 \Delta ) \eta _u \vert
\eta_u \  {\rm  as \ in \ the \ definition \ of } \  D (A_\delta) \}.$
Obviously,  $\Gamma$ is convex since $j$ is convex, and
 $\Gamma$ is proper since $\mathcal D (\Gamma)$  is non-empty
and even dense in $ H^{-1}$, because $L^2(\R) \subset \shd(\Gamma)$.
The rest of the  proof will be done in  a series of lemmas.
\begin{lemma}\label{L101}
 The function $\Gamma$ is lower semicontinuous.  
\end{lemma}
\begin{proof}{}
   First of all we observe that $\Gamma$ is lower semicontinuous 
on $L^1 _{\text{loc}}(\R)$.
   In fact, defining $\Gamma_N, N \in \N$, analogously to $\Gamma$,
with $j \wedge N$ replacing $j$, by the continuity of $j$ and Lebesgue's 
dominated converegence theorem, $\Gamma_N$ is continuous in 
$L^1 _{\textnormal{loc}}$. Since $\Gamma = \sup_{N \in \N} \Gamma_N,$
it follows that $\Gamma$ is lower continuous on  $L^1 _{\textnormal{loc}}$.
Let us suppose now that $u_n \rightarrow u$ in $H^{-1}(\R)$.
We have to prove that
  \begin{equation}\label{E102}
   \int_\R  j(u(x)) dx \leq \liminf _{n \rightarrow \infty}
 \int _\R j (u_n(x)) dx.
  \end{equation} 
 Let us consider a subsequence such that
$\int j (u_n(x)) dx $ converges to the right-hand side
 of \eqref{E102} denoted by $C$.
We may  suppose $C < \infty$. According to \eqref{ASS}, we have
\[\lim_{R \rightarrow \infty} \frac {j (R)} R = \infty,\]
which implies 
that the sequence $(u_n)_{n \in \N}$ is uniformly integrable on
$[-K, K]$ for each $K>0$.
Hence, by Dunford-Pettis theorem, the sequence $(u_n)$ is
 weakly relatively compact in $L^1 _{\text{loc}}$.
Therefore, there is a subsequence $(n_l)$ such that $(u_{n_l})$ converges weakly in $L^1 _{\text{loc}}$, necessarily to $u$,
since $u_n \rightarrow u$ strongly, hence also weakly in $H^{-1}(\R)$.
Since $\Gamma$ is convex and  lower semicontinuous on $L^1 _{\textnormal{loc}}$,
 it is also weakly lower semicontinuous on  $L^1 _{\textnormal{loc}}$,
sse \cite{choquet} p.62, 22.1.
This implies that
\[\int j(u(x)) dx \leq \liminf _{l \rightarrow \infty }\int j (u_{n_k} 
( x )) dx = C.\]
Finally, \eqref{E102}and thus the assertion of Lemma \ref{L101} is 
proved.
\end{proof}

An important intermediate step is the following.
\begin{lemma}\label{L102}
   $\mathcal D (\Gamma ) = D (A_\delta ) $ and $\partial _H \Gamma(u) = 
A_ \delta u,\, \forall u\in \mathcal D (\Gamma)$.
In particular $ D (A_\delta ) $ is dense in $H^{-1}$. \\
We observe, that $\partial _H$ depends in fact on $\delta$
since the inner product on  $H^{-1}$ depends on $\delta$.
\end{lemma}

\begin{proof}{}
   Let $u \in \mathcal D (\Gamma) $, $h \in L^2 (\subset \mathcal D (\Gamma)$).
   For $z \in \partial _H \Gamma (u) $ we have
   \begin{equation}\label{B11}
      \Gamma(u + h ) - \Gamma (u) \geq \bracket{ z , h }_{-1 , \delta }= \int _\R v(x) h (x) dx,
   \end{equation}
   where $v = (\delta - \frac 12 \Delta ) ^{-1}z$.
 Clearly $v \in H^1 $. By \eqref{B11} it follows that
   $v\in \partial _{L^2 }\tilde \Gamma (u)$ where $\tilde \Gamma$ is 
the restriction of $\Gamma$ to $L^2 (\R)$.
   By Example 2B of Chapter IV.2 in \cite{sho97}, this  yields that
 $v \in \beta (u)$ a.e. 
Consequently, $\mathcal D(\Gamma) = D (A_\delta)$
   and $\partial _H \Gamma(u) \subset A_\delta u,
 \forall u \in \shd(\Psi)$.
   
   It remains to prove that $A_\delta (u) \subset \partial _H \Gamma (u), 
\forall u \in \shd(\Gamma)$.
 Let $u \in \mathcal D (\Gamma)$, $h \in L^2 $,
   $\eta _u\in \beta (u)$ a.e. with $\eta_u\in H^1$. Since
   \[j (u + h ) - j(u) \geq \eta _u  h\quad \text{a.e.},\]
   it follows
   \begin{equation}\label{BB1}
      \Gamma (u+ h ) - \Gamma (u) \geq \int _ \R \eta _u(x ) h (x ) dx = 
\bracket{ (\delta - \frac 12 \Delta)\eta _u , h }_{-1,\delta}.
   \end{equation}
   It remains to show that \eqref{BB1} holds for any
 $h \in H^{-1}$  such that $ u + h \in \mathcal D (\Gamma)$.
 Then we have
 $u + h, $ $u \in L^1  _{\text{loc}}$
   and $j(u) , j(u+ h)\in L^1$. We first prove that
 \eqref{BB1} holds if
 $h \in L^1 (\subset H ^{-1})$.
   We truncate $h$ setting
   \[h_n  = 1_{\{\vert h \vert \le n\}} h, n \in \N, \]
   so that $h_n\in L^2 (\R)$. Now
   \[j (u + h_n )(x) = \begin{cases}
                        j (u(x) +h(x) ) &{\rm if}\  \vert h(x) \vert \leq n,\\
                        j (u(x))      &{\rm if}   \   \vert h(x)\vert > n,
                       \end{cases}
\]
and it is dominated by
\[j (u + h ) + j (u ) \in L^1.\]
We have
\begin{equation}\label{BB3}
   \int _{\R} j ( u + h_n ) (x) dx \geq \int _\R j ( u (x)) dx + 
\bracket{(\delta - \frac 12 \Delta) \eta _u , h _n }_{-1, \delta}.
\end{equation}
Since $h_n \rightarrow h$ in $L^1 $ (and so in $H^{-1}$), using Lebesgue's
 dominated convergence theorem, \eqref{BB1} follows for 
$h \in L^1$.

Let $M > 0$ and consider a smooth function $\chi: \R \rightarrow [0,1]$
 such that $\chi (r) = 1 $ for $ 0 \leq \vert r \vert \leq 1$, $\chi (r)=0$
for $2 \leq \vert r \vert < \infty$. We define
\[\chi _M (x) = \chi \left(\frac{x}{M}\right), \quad x \in\R. \]
Then
\[\chi _M (x) = \begin{cases}
                  1 : \quad & |x| \le M\\
                  0: \quad & | x | \ge 2M.
                \end{cases}
\]
Since $h \chi _M \in L^1$,  we have
\begin{equation}\label{BB11}
   \int_\R (j (u + h \chi _M ) (x)  - j(u)(x)) dx \geq 
\bracket{(\delta - \frac 12 \Delta )\eta _u, h \chi_M }_{-1,\delta}.
\end{equation}
Since $j$ is convex  and non-negative, we have
\begin{eqnarray*}
j ( u + h \chi _M) &= & j((1- \chi_M) u + \chi_M (u+h)) \\
&\leq& 
(1- \chi_M) j(u) + \chi_M j(u+h) 
\leq 
 j(u) + j (u+ h).
\end{eqnarray*}
Hence Lebesgue's domintated convergence theorem allows to take
 the limit in the left-hand side, when $M \rightarrow \infty$ of
\eqref{BB11} to obtain
\[\int _\R ( j(u+ h)(x)) - j(u(x)) dx .\]
The right-hand side of \eqref{BB11} converges to 
$\bracket{(\delta - \frac 12 \Delta) \eta_u, h} _{H^{-1}}$ because of the
next lemma. Hence, the assertion of Lemma \ref{L102} follows.
\end{proof}

\begin{lemma}\label{BB5} 
Define $h_M := \chi_M h$ in $H^{-1}, M > 0$. Then
 \[\lim_{M \rightarrow \infty} h _M = h \quad {\rm weakly} \
\text{in } H^{-1}.\]
\end{lemma}
\begin{prooff}\ (of Lemma \ref{BB5}).
 Let us first show that the
 sequence $(h_M)$ is bounded in $H^{-1}$.

      In fact, given $\varphi \in H^1$,
      \begin{align*}
         \left| \int _\R h_M ( x ) \varphi (x) dx \right| &=
 \left|\int _{\R} h(x) \chi _M(x) \varphi (x) dx \right|\\
         & \leq \  \norm {h }_{H^{-1}}  
\left(  \delta^{\frac 12} \norm{  
\chi _M \varphi }_{L^2} +  \norm {\chi_M' \varphi
  + \varphi' \chi_M}_{L^2}\right)\\
         & \leq \norm {h }_{H^{-1}} \left(\delta^{\frac 12}
 \norm{ \varphi }_{L^2} 
+ \norm {\varphi'
}_{L^2}  + \frac{\norm{\chi'}_{\infty}}{M}\norm \varphi_{L^2}  \right)\\
         & \leq \text{const } \norm{h}_{H^{-1}} \norm{\varphi}_{H^1}
      \end{align*}
for some  positive constant independent of $M$. 

Hence there is a subsequence weakly converging to some $k \in H^{-1}$.
Since
\[\int _\R  h_M(x) \varphi (x) dx \rightarrow_{M \rightarrow \infty}
 \int _\R h (x) \varphi (x) dx \]
for any $\varphi \in C_0^\infty(\R)$, $k$ must be equal to $h$.
Now the assertion of Lemma \ref{BB5} follows.
\end{prooff}
By Corollary IV 1.2 in \cite{sho97}, we know that
$A_\delta $ is maximal monotone  on $H^{-1}$
 and therefore $m$-accretive 
with domain $D (A _ \delta ) = \mathcal D(\Gamma)$.

We go on with the proof of Proposition \ref{L4.3a}.
 Since our initial condition $u_0$ belongs to $L^1 \cap L^\infty$ and
$L^2\subset  \mathcal D (\Gamma)$, clearly $u_0\in \mathcal D (A_ \delta )$.
According to Komura-Kato theorem, see \cite[Proposition
  IV.3.1]{sho97},
 there exists a (strong) solution
$ u = u_\delta : [ 0,T] \to E = H^{-1}$ 
 of
\begin{equation}\label{EA4}
   \begin{cases}
    \frac {d u } {dt} + A_\delta u \ni 0, \quad t \in [0,T]\\
    u(0,\cdot) = u_0,
   \end{cases}
\end{equation}
 which is Lipschitz.
   In particular,
 for almost all $t \in [0,T]$,
$u_\delta(t,\cdot) \in \mathcal D (A_\delta )$
and there is $\xi_\delta (t,\cdot) \in H^1$ such that
 $\xi_\delta (t,\cdot)
\in \beta (u_\delta (t,\cdot))$
a.e., 
$t \mapsto (\delta  \xi_\delta - \frac 12 \Delta \xi_\delta)(t,\cdot)
\in H^{-1}$ is measurable  and
\begin{equation}\label{EA5}
   u_\delta(t,\cdot) = u_0 + \int_0^t \left(\delta \xi_\delta  - \frac 12 
\Delta \xi_\delta
 \right)
 (s,\cdot ) ds 
\end{equation}
in $H^{-1}$.

Furthermore, for the right-derivative $D^+ u_\delta(t)$, we have
\begin{equation}\label{EA10}
D^+ u_\delta(t,\cdot)) + 
(A_{\delta}) ^\circ u_\delta(t,\cdot) = 0 \quad {\rm in} \ 
  H^{-1}  ,\forall t \in [0,T],
\end{equation}
where $(A_\delta )^\circ$ denotes the minimal section of $A_\delta$
and the map 
$t \mapsto \norm{(A_\delta )^\circ u_\delta (t, \cdot)}_{-1, \delta }$ 
is decreasing.
On the other hand \eqref{EA5} implies that
\begin{equation}\label{EA11}
   \frac{du_\delta}{dt}(t,\cdot) + \delta \xi_\delta (t,\cdot) - 
\frac 12 (\xi_\delta ) '' (t,\cdot) = 0 \quad 
{\rm for} \ {\rm a.e.} \ t \in [0,T].
\end{equation}
Consequently, for almost all $t \in [0,T]$ 
\begin{equation}
 \norm{\delta \xi_\delta (t,\cdot) -  \frac 12
 (\xi_\delta ) '' (t,\cdot)}_{-1,\delta}
 = \norm{(A_\delta )^\circ  u_\delta(t,\cdot)}_{-1, \delta }
 \leq \norm{(A_\delta )^\circ u_0}_{-1, \delta },
\end{equation} 
i.e. setting $\xi_0 = (\delta - \frac 12 \Delta)^{-1} 
(A_\delta)^\circ u_0 ,$ we observe that it belongs to $H^1$ and that
$$ \prescript{} 
{H^{-1}}{\bracket{(\delta - \frac 12 \Delta ) \xi_\delta (t,\cdot), 
\xi_\delta (t,\cdot)}}_{H^{1}} 
\le \quad  \prescript{}{H^{-1}}
{\bracket {(\delta - \frac 12 \Delta) \xi_0, 
\xi_0}}_{H^{1}},$$
for a.e. $t \in [0,T]$.

Consequently, for a.e. $t \in [0,T]$,
\begin{eqnarray} \label{EA13}
\int _\R (\delta \xi_{\delta} (t,x) ^2 &+&  \frac 12 \xi_\delta'
 (t,x) ^2 ) dx \nonumber \\ & \leq &
\delta \int _\R \xi _0 ^2 (x) dx + \int _\R {\xi _0'}^2 (x) dx \\
 &\le & \Vert \xi_0 \Vert_{H^1}    =: C \nonumber
\end{eqnarray}
since $\delta \leq 1$.

We now consider  equation \eqref{EA4} from 
an $L^1$ perspective, similarly as for equation \eqref{PME},
see Proposition \ref{R4.1} 2. Since our initial
condition $u_0$ belongs to $(L^1 \cap L^\infty)(\R)$, equation
\eqref{EA4}
 can also be considered as an evolution problem
on the Banach space $E = L^1 (\R)$. More precisely define
$$ D(\tilde A_\delta):= \{u \in L^1(\R) \vert \exists  w \in L^1_{\rm loc}:
w \in \beta(u) \ { \rm a.e. \ and \ }  (\delta  - \frac 12 \Delta) w
\in L^1(\R) \}   $$ 
and for $u \in  D(\tilde A_\delta)$,
$$ \tilde A_\delta u : = \{ (\delta  - \frac 12 \Delta) w \vert w
\ {\rm as \ in } \ D(\tilde A_\delta) \}.$$
Note that for $w$ as in the definition of $D\tilde( A_\delta)$,
we have $(\delta  - \frac 12 \Delta)w \in H^{-1}$,
since $L^1(\R) \subset H^{-1}$.
Therefore, $w \in H^1$, hence
\begin{equation} \label{E425prime}
 D(\tilde A_\delta) \subset  D(A_\delta)\ {\rm and } \
 \tilde A_\delta =  A_\delta  \ {\rm on} \   D(\tilde A_\delta).
\end{equation}
Furthermore, as indicated in Section 3, 
it is possible to show that
$ \tilde A_\delta$ is an $m$-accretive operator on $L^1$.

For $\lambda > 0$, the following four points are then
 a consequence of Remark \ref{R31bis} and 
Lemma \ref{L31}.
\begin{enumerate}
   \item  For each $f \in L^1 (\R) $ there is $u\in L^1$, $w\in L^1$ with 
$w\in \beta (u)$ a.e. and
   \begin{equation}\label{EA14}
      u +  \lambda(\delta  w - \frac 12 \lambda w'') = f.
   \end{equation}
  \item The map 
\begin{equation} \label{E426prime}
f \mapsto u : = (I + \lambda  \tilde A_\delta)^{-1}(f) 
\ {\rm  is \ a \ contraction \ on}  \ L^1.
\end{equation}
   \item $\overline{ D (\tilde A_\delta )} = L^1$.
   \item We recall that whenever $f \in L^\infty$, then 
$u \in L^\infty$ and
   \begin{equation}\label{EA15bis}
  \norm{u}_\infty \leq \norm f _\infty .
   \end{equation}
\end{enumerate}
   Therefore, there is a $C^0$-solution
   $\tilde u : [0,T]\times \R \to \R$ of \eqref{EA4}.
Since by \eqref{E425prime}, every $\varepsilon$-solution
of \eqref{EA4} in $L^1(\R)$ is also an $\varepsilon$-solution
of  \eqref{EA4}
 in $H^{-1}$
and $L^1 \subset H^{-1}$ continuously, $\tilde u$ is also 
a $C^0$-solution of \eqref{EA4} in  $H^{-1}$.
 Since, by Proposition IV 8.2 and 8.7
of \cite{sho97}, the solution above is the unique $C^0$-solution of 
\eqref{EA4} in $H^{-1}$, we have proved the first part of the following 
lemma.

\begin{lemma}\label{LEA1}
 The solution $\tilde u$ coincides with the $H^{-1}$-valued solution
 $u_\delta$. Moreover, for $p = 1$ or $p = \infty$ and $c$ as in 
Hypothesis \ref{H3.0}
 \begin{equation}\label{EA16}
   \sup_{t \leq T} \norm{ u_\delta(t,\cdot) }_{L^p}
\leq \norm {u_0}_{L^p} \ {\rm and } \
   {\rm esssup}_{t \leq T} \norm{ \xi_\delta(t,\cdot)}_{L^p}
\leq c \norm {u_0}_{L^p}
 \end{equation} 
\end{lemma}
\begin{proof}\ 

It remains to show  \eqref{EA16}. 
As in the proof of \eqref{E4.1prime} by
 \eqref{E426prime}, 
\eqref{EA15bis} and induction,  we easily obtain that
for any $\varepsilon$-solution in $L^1$ and $p = 1$ or $p = \infty$,
\[\sup_{t \le T}  \norm {u^\varepsilon (t,\cdot)}_{L^p}
 \leq \norm{u_0}_{L^p}.  \]
 The conclusion follows because for every $t \in [0,T]$, there
 is a sequence $(\varepsilon _n)$ such that
$u^{\varepsilon_n} (t, \cdot) \to \tilde u (t , \cdot) =
u_\delta(t,\cdot)$ a.e. as $n \rightarrow \infty$.
The second part of \eqref{EA16} then obviously follows by
Hypothesis \ref{H3.0}, since 
$\xi_\delta(t,\cdot) \in \beta(u_\delta(t,\cdot))$ a.e.
for a.e. $t \in [0,T]$.

\end{proof}


\begin{lemma}\label{LA11}
 We have $u_\delta \to u$ in $C ( [0,T]; L ^1 (\R))$ 
as $\delta \rightarrow 0$, where $u$ is the solution to \eqref{PME}.
\end{lemma}
\begin{proof}{}\
It will be enough to prove that for $\delta$ small enough,
we have
\begin{equation} \label{EDisdelta}
\int _\R | u_\delta (t,x) - u(t,x) | dx \leq c T \norm{u_0}_{L^1} 
 \delta.
\end{equation}
Using point 5. of Proposition \ref{R4.1} in a slightly modified
form, and approximating $\beta$ by
$\beta^\varepsilon (u) = \beta(u) + \varepsilon u$, it is 
enough to suppose that $\beta$ is strictly monotone,
i.e.  \eqref{(3.9)second} holds.
In the lines below the parameter $\varepsilon$ will play
however a different role.\\
   We need to go back to the $L^1$-$\varepsilon $-solutions
 related to $u_\delta $ and $u$. 

   For $\varepsilon >0$ we consider a subdivision 
$0 = t _0 ^\varepsilon < \ldots < t _j ^\varepsilon < \ldots
  < t _N^\varepsilon =T$
   such that $t _j^{\varepsilon } - t _{j-1}
^\varepsilon < \varepsilon $, $j = 1, \ldots, N$.
  Similarly as in Lemma \ref{R4.2}
   \begin{equation}\label{EA22}
   \begin{split}
      u^{\varepsilon }_ \delta (t_j ,\cdot)&= 
u^{ \varepsilon }_ \delta (t _{j -1}, \cdot)\\
      & + (t _j - t _{j-1}) \frac{(\eta^{ \varepsilon}_\delta)''}
{2}(t_j , \cdot)\\
      & - ( t _j  - t _{j -1}) \delta \eta ^{ \varepsilon}_\delta(t_j , \cdot)
      \end{split}
   \end{equation}
and 
  \begin{equation}\label{EA23}
      u^\varepsilon(t_j ,\cdot ) = u ^{\varepsilon}(t_{j -1},\cdot) + 
 (t_j - t _{j-1})\frac 12 (\eta ^\varepsilon)'' (t_ j ,\cdot)
   \end{equation}
   with $\eta ^{ \varepsilon}_\delta \in \beta 
( u^{\varepsilon}_\delta)$,
 $\eta^\varepsilon \in \beta ( u ^\varepsilon )$ a.e..
   Taking the difference of the previous two equations we obtain
   \begin{eqnarray}\label{EA24}
       u ^{ \varepsilon}_\delta (t _ j , \cdot)
 - u ^\varepsilon (t_ j , \cdot)  \nonumber
       &=&  u^{\varepsilon } _\delta (t _{j-1}, \cdot) -
 u^\varepsilon (t_{j-1}, \cdot)   \\
&& \\
       & +& \left(\frac{t _ j - t_{j-1}}{2}  \right)
(\eta ^{ \varepsilon}_\delta - \eta ^{\varepsilon})'' (t_ j , \cdot)
       -\delta (t _ j- t_{j-1}) \eta ^{\varepsilon}_\delta(t_j , \cdot).
\nonumber  
   \end{eqnarray} 
Let $\Psi_\kappa: \R \to [-1,1]$ be an odd smooth increasing 
function such that $\Psi_\kappa (x ) \to 
{\sign} \ x $  as $\kappa \rightarrow 0$  pointwise,
We integrate \eqref{EA24} against 
$\Psi_\kappa (\eta ^{ \varepsilon}_\delta(t_ j , \cdot) -
 \eta ^\varepsilon (t _j, \cdot))$ and we get
\begin{align*}
 &\int _\R ( u ^{ \varepsilon}_\delta  ( t _j , x) -
 u^ \varepsilon (t_j ,x) ) \Psi_\kappa ( \eta ^{\varepsilon}_\delta
(t _j , x ) - \eta ^\varepsilon ( t_ j , x )) dx\\
 =& \int _\R ( u ^{\varepsilon}_\delta ( t _{j-1} , x) - 
u^ \varepsilon (t_{j-1},x) ) \Psi_\kappa
 ( \eta ^{ \varepsilon}_\delta (t _j , x ) - 
\eta ^\varepsilon ( t_ j , x )) dx   \\
 &- \frac{( t _ j - t_{j-1})}{2} 
\int _\R ( \eta ^{\varepsilon}_\delta  -
 \eta ^{\varepsilon})'(t_j , x)^2 \Psi_\kappa ' 
( \eta ^{\varepsilon }_\delta (t_ j ,x )
 - \eta ^\varepsilon (t_ j, x)) dx \\
 &- \delta ( t _ j - t_{j-1}) \int _\R ( \eta ^{\varepsilon}
_\delta (t_ j ,x) \Psi_\kappa ( \eta ^{\varepsilon} _\delta(t _j , x) 
- \eta ^{\varepsilon}(t_ j ,x)) dx .
\end{align*}
Using the fact that $\Psi_\kappa ' \geq 0$, $| \Psi_\kappa | \leq 1$,
that,
by strict monotonicity of $\beta$
\[\sign( \eta ^{ \varepsilon }_\delta (t_ j , \cdot) -
 \eta ^\varepsilon (t_j , \cdot)) = 
 \sign ( u^{\varepsilon }_\delta (t_j , \cdot) - 
u^{\varepsilon}(t_ j ,\cdot)) ,\]
a.e. on $ \{u^{\varepsilon }_\delta (t_j , \cdot)
\neq u^{\varepsilon}(t_ j ,\cdot)\}$, 
 and letting $ \kappa \to 0$, by \eqref{E4.1prime},
we obtain
\begin{equation}\label{EA32}
\begin{split}
   &\int _\R | u^{\varepsilon}_\delta (t _ j , x) - u^\varepsilon (t _j ,x )| dx\\
   \leq& \int _\R | u^{\varepsilon}_\delta
 (t _{j-1},x) - u ^\varepsilon (t _{j-1},x)| dx
+  c \delta (t _j - t_{j-1}) \norm{u_0}_{L^1} .
\end{split}
\end{equation}
Since $\int _\R | u^{\varepsilon}_\delta
(0, x ) - u ^{\varepsilon }(0, x) | dx =0$, an induction argument implies that
\[\int _\R | u ^{ \varepsilon}_\delta (t_j ,x ) -
 u^\varepsilon (t_j ,x)| dx \leq  c T \norm{u_0}_{L^1}  \cdot \delta\]
for every $j \in \{0, \ldots, N\}$.
Consequently, for any $t \in [0,T]$
\[\int _\R | u^{\varepsilon}_\delta(t,x) - 
u ^\varepsilon (t,x)| dx \leq c T \norm{u_0}_{L^1}  \delta.\]
Letting $\varepsilon \to 0$, \eqref{EDisdelta} follows and
 Lemma \ref{LA11} is proved.
\end{proof}

By \eqref{EA5},  for every $\alpha \in C_0^\infty(\R)$
and all $t \in [0,T]$,
\begin{align*}
 \int_\R  u_\delta ( t ,x)\alpha (x) dx &= \int_\R u_0 (x) \alpha (x) dx \\
   & - \delta 
\int _0 ^t ds \int _\R \xi_\delta (s,x) \alpha (x) dx + \frac 12
 \int _0^t ds \int_\R dx 
\xi_\delta (s, x) \alpha '' (x),
\end{align*}
$\xi_\delta \in \beta (u_\delta)$ a.e.
Letting $\delta \to 0$,  by \eqref{EA16} and Lemma \ref{LA11},
 we obtain that
\begin{equation}\label{EA41}
   \int _\R u (t,x) \alpha (x) dx = \int_\R u _0 (x) \alpha (x) dx
 + \frac 12 \lim _{\delta \to 0} \int _0^t ds \int _\R
 \xi_\delta (s,x) \alpha''(x) dx.
\end{equation}
By \eqref{EA16} it follows that for each $K  > 0$,
$u_\delta \to u$ in $L^2 ( [0,T] \times [ - K , K ] )$ 
and that
$(\xi_\delta)$,
is bounded in $L^2 ([0,T] \times [-K,K] )$.
Since, by \cite{sho97}  Example IV.2C, 
 the map $u \mapsto \beta (u)$ is  $m$-accretive
on $L^2 ([0,T] \times [-K,K] )$, it is weakly-strongly closed,
see \cite{Barbu1}, p.37 Proposition 1.1 (i) and (ii).
 So, there is a sequence $(\delta _n)$ such that $\xi_{\delta_n}\to \xi$
 weakly in $L^2 ( [0,T] \times [-K, K])$ for some $\xi \in \beta (u)$ a.e.
Hence, \eqref{EA41} implies
\begin{equation}\label{EA42}
   \int _\R u (t,x) \alpha (x) dx = \int _\R u _0 (x) 
\alpha (x) dx + \frac 12 \int _0^t ds \int _\R \xi (s, x) \alpha '' (x) dx.
\end{equation}
By the uniqueness part  of Proposition \ref{R4.1} 1., we conclude 
that $\xi \equiv \eta_u$.

\bigskip
By Proposition \ref{P4.3}, we already knew that
 $\eta _u(t, \cdot) \in H^1 (\R)$ for almost any $t$.
By \eqref{EA13} for a.e. fixed $t$, there is a sequence $(\delta_n )$ 
such that $( \xi_{\delta _n})(t, \cdot)$ weakly
converges to some $\tilde \xi ( t, \cdot)$ in $H^1 (\R )$
hence in $L^2 (\R )$.

Consequently $\tilde \xi ( t , \cdot) =
 \eta _u  ( t,\cdot)$ for almost all $t \in [0,T]$.

Recalling \eqref{EA13}
for a.e. $t\in [0,T]$ we get
\[\int _\R dx\;  \eta _u ' (t,x) ^2 = \int _\R dx \; \tilde \xi' (t,x) ^2
\leq \liminf _{\delta \to 0} \int _\R dx \;(\xi_\delta )' (t,x ) ^2 \leq C.\]
This finally completes the proof of Proposition \ref{L4.3a}.
\end{proof}


At this point, we can state and prove the following important theorem.
\begin{theo}\label{L4.7}
Assume that Hypothesis \ref{H3.0}
and condition \eqref{ASS} hold.
Let $u$ be the solution of \eqref{PME} (or equivalently of \eqref{E4.111},
from Proposition \ref{R4.1}). Then
 the function $t \mapsto \int _\R j (u(t,x)) d x $ is absolutely continuous.
\end{theo}

\begin{proof}{}
Let $0 \le s <      t \le T$. 
Let $\Gamma$ and $\shd(\Gamma)$ be as defined in   the proof of
Proposition \ref{L4.3a}.
Since $u(t,\cdot) \in \shd(\Gamma)$ for a.e. $t \in [0,T]$,
 Lemma \ref{L102} applies and thus for a.e. $t, s \in [0,T]$
we have that
$(\delta - \frac 12 \Delta) \eta_u(t,\cdot)  \in A_\delta u(t,\cdot)$,
  and
\begin{eqnarray*}
\vert \Gamma(u(t,\cdot)) - \Gamma(u(s,\cdot)) \vert & \le & 
 \max_{r \in \{t, s\}} \vert <  (\delta - \frac 12 \Delta ) \eta_u(r, \cdot),
  u(t,\cdot) - u(s, \cdot)>_{-1,
  \delta} \vert \\ &\le&  \max_{r \in \{t, s\}}
 \Vert  (\delta - \frac 12 \Delta  ) \eta_u(r, \cdot) \Vert_{-1, \delta} 
 \Vert u(t,\cdot) - u(s,\cdot) \Vert_{-1, \delta} \\
&\le & \left( {\rm esssup}_{r \in [0,T]} \sqrt{\delta \int_\R \eta_u(r, x)^2 dx
+ \frac 12 \int_\R \eta_u'(r, x)^2 dx} \right) \\
&& \Vert u(t,\cdot) - u(s,\cdot) \Vert_{-1, \delta}.
\end{eqnarray*}
 By \eqref{E4.12prime} and \eqref{E3.0}, this is bounded by
$$ {\rm max}(c,C)  \sqrt{\delta \Vert u_0 \Vert_{L^2}^2  + 1}
 \Vert u(t,\cdot) -
u(s,\cdot) \Vert_{-1, \delta},$$
where
we recall that by Remark \ref{R00} the map $t \mapsto u(t,
\cdot) $ is absolutely continuous in $H^{-1}$.
Since by Proposition \ref{P4.3} c), $t \mapsto \Gamma(u(t,\cdot))$
is continuous, we have 
$$ \vert \Gamma(u(t,\cdot)) - \Gamma(u(s,\cdot)) \vert
\le {\rm const} \Vert u(t,\cdot) -
u(s,\cdot) \Vert_{-1, \delta}, \ \forall t,s \in [0,T], $$
and the assertion follows. 
\end{proof}

We are now prepared to prove the first main result of this section,
which will be used in the next section in a crucial way.
\begin{theo}\label{T4.5}
 Under Assumption \eqref{ASS}, the unique solution to  \eqref{PME} verifies
 \begin{equation}\label{4.10bis}
   \int_\R j(u(t,x)) dx = \int _\R j(u(r,x)) dx -
 \frac 12 \int _r ^t ds \int _\R {\eta_u'}^2 (s,x)dx 
 \end{equation}
 for every $0\leq r\leq t \leq T$.
\end{theo}
\begin{proof}  \ 
 For a.e. $t\in [0,T]$,  \eqref{4.11bis} gives
 \begin{equation*}
   \bracket{\frac d {dt} u(t,\cdot), \varphi } =
 \frac 12 \bracket{\eta_u(t,\cdot), \varphi ''},\quad
   \forall \varphi \in C_0^\infty(\R).
 \end{equation*}
 By density arguments,
 $$ _{H^{-1}} \bracket{\frac d {dt} u(t,\cdot) ,
 \psi}_{H^1} = -\frac 12 \int _\R \eta _u' (t, x) \psi'(x) dx  $$
 for every $\psi \in H^1 (\R)$. For $\psi= \eta_u (t,\cdot)$, we get
 \begin{equation}\label{4.12}
 _{H^{-1}}\bracket{\frac d {dt} u(t,\cdot) , \eta _u(t, \cdot) }_{H^1} 
= - \frac 12
 \int _\R {\eta_u'}^2 (t, x) dx.
 \end{equation}
 Since $u \in (L^1 \bigcap L^\infty) 
([0,T] \times \R)$ and $|j(u)| \leq c | u | ^2,$ then, in particular,
it belongs to $ L^2( [0,T], L^2(\R))$. We need the following lemma.
\begin{lemma}\label{L4.5}
 For a.e. $t \in [0,T]$
 \begin{equation}\label{4.13}
  _{H^{-1}}\bracket{\frac d {dt} u (t,\cdot) , \eta_u(t, \cdot)}
_{H^1}  = \frac d {dt} \int _\R j (u(t,x)) dx .
 \end{equation} 
\end{lemma}
\begin{proof} \
 Let $t \in ]0,T]$ such that
 \[\frac {u(t + h , \cdot) - u(t,\cdot)} h 
\stackrel{h \to 0}\longrightarrow \frac d{dt} u(t,\cdot)
 \quad \text{in }H^{-1}(\R).\]
Let $h > 0$ such $t - h, t+h$ are both positive.  We have by \eqref{BB1}
\begin{align*}
 \int _\R \frac {j(u(t,x)) - j (u(t-h, x))}h dx \leq 
\bracket{ \frac{u(t,\cdot) -u(t-h, \cdot)}{h} , \eta_u(t, \cdot)}_{L^2}.
\end{align*}
Taking  limsup for $h \to 0$, we get
\begin{equation}
 \limsup_{h \to 0} \int _\R \frac {j(u(t,x)) - j (u(t-h, x))}h dx
 \leq    _{H^{-1}} \bracket{ \frac d {dt} u(t,\cdot) , \eta_u(t, \cdot)}_{H^1}.
\end{equation} 
On the other hand
\begin{align*}
 \bracket{\frac {u(t + h , \cdot ) - u(t, \cdot)}h , \eta _u (t,\cdot)}
_{L^2} \leq \int_\R  \frac {j (u(t+ h , x)) - j (u(t,x))}{h} dx.
\end{align*}
So
\begin{align*}
 _{H^{-1}} \bracket{\frac d{dt} u (t,\cdot ) , \eta_u(t, \cdot)}_{H^1}
 \leq \liminf _{h\to 0 } \int _\R \frac {j(u(t+ h, x)) - j(u(t,x)) }h dx.
\end{align*}
Consequently for  a.e. $t \in [0,T]$,
\begin{equation}\label{4.14}
\begin{split}
 &\limsup_{h\to 0}\int _\R \frac {j (u(t,x)) - j(u(t-h,x))} h dx
 \leq  _{H^{-1}}
 \bracket {\frac {du }{dt}(t, \cdot) ,
 \eta_u(t,\cdot )} _{H^{1}} \\
 \leq &\liminf _{h \to 0} \int _\R \frac {j (u(t+ h ,x)) - j(u(t,x))} h dx.
 \end{split}
\end{equation}
On the other hand we know already by Theorem \ref{L4.7} that 
 for  a.e. $t \in [0,T]$, the limsup and liminf-terms in \eqref{4.14}
coincide. Hence the assertion follows.

\end{proof}

At this point \eqref{4.12} and Lemma \ref{L4.5} imply that 
 for  a.e. $t \in [0,T]$,
\begin{equation}\label{4.16}
 \frac d {dt} \int _\R j (u(t,x)) dx =
 - \frac 12 \int _\R { \eta_u'}^2  (t,x) dx.
\end{equation}
Theorem \ref{L4.7} says that $t \mapsto \int _\R j (u(t,x)) dx $ is
absolutely continuous. So, after integrating in time, we get
\begin{equation} \label{4.17a} 
\int _\R j(u(t,x)) dx = \int _\R j (u(r,x))dx - 
\frac 1 2 \int _r^t ds \int _\R \eta '_u (s, x)^2 dx .
\end{equation}
This completes the proof of Theorem \ref{T4.5}.
\end{proof}

The second main result of this section, also crucially used in 
Section 5 below, is the following.

\begin{prop}\label{P4.8}
Let Hypothesis \ref{H3.0} hold and
 let $u$ be the unique solution to \eqref{PME} with initial
 condition $u_0 \in L^1 \bigcap L^\infty$ being locally  of bounded variation.
    Then, for each $t \in [0,T]$, $u(t,\cdot)$ also has  locally 
bounded variation. 
\end{prop}

\begin{rem}\label{R4.9}
\begin{enumerate}
\item We note that \eqref{ASS} is not needed for the above proposition.
\item  
 Since $u(t,\cdot)$ has locally bounded variation,  it has
 at most a countable
   number of discontinuities. We will see that in the
 degenerate case, i.e.
   if $\Phi(0)= 0$, a suitable section of
 $\Phi(u(t,\cdot ))$,  also has at most 
countably  many discontinuities, see Lemma \ref{L4.10} below.
\end{enumerate}
\end{rem}
\begin{prooff}\ (of Proposition \ref{P4.8}).
\
For $h$ small real fixed, we set
 \[u^h(t,x) = u(t,x+h ) - u(t,x) .\]
Let $\zeta$ be a smooth nonnegative function with compact support on some
compact interval. 
We aim at establishing the following intermediate result:
\begin{equation} \label{EBV10}
  \int _\R \zeta (x) | u^h (t,x)| dx \leq
     \int _\R \zeta (x)  |(u_0)^h (x)| dx
 +   c \Vert \zeta''' \Vert_\infty \vert h \vert
   \int _{[0,T]\times \R} | u(s,x)| ds dx.
\end{equation}
  Approximating $\beta$ with $\beta^\varepsilon$ as in 
 Proposition \ref{R4.1} 5., we may suppose that $\beta$ 
satisfies \eqref{(3.9)second} on $\beta$.
In the rest of this proof $\varepsilon$ will however be
the discretization mesh related to an $\varepsilon$-solution.
We recall that $u$ is the unique $C^0$-solution to \eqref{PME}.
 So for fixed $t \in ]0,T]$
 \begin{equation}\label{4.50}
   u(t,\cdot) = \lim_{\varepsilon \to 0} u^{\varepsilon } 
(t,\cdot ) \quad\text{in }L^1,
 \end{equation}
 where $u^\varepsilon (t,\cdot)$ is given in Lemma \ref{R4.2}.

According to Lemma \ref{L33} we have, for $i=1,\ldots, N$,
\begin{eqnarray*} 
\int _\R \zeta (x) |u^{h}_i(x)| dx &\le& 
  \int _\R \zeta (x)  u_{i-1}^h(x) \sign(w^h_i(x))  dx
 +   c \Vert \zeta''' \Vert_\infty
  \vert h \vert \varepsilon \int _\R | u_i (x) | dx\\ 
 &\le& 
\int _\R \zeta (x)  \vert u_{i-1}^h(x) \vert   dx
 +   c \Vert \zeta''' \Vert_\infty
  \vert h \vert \varepsilon \int _\R | u_i (x) | dx,
\end{eqnarray*}
 where $u^{h}_i = (u_i)^h, w^{h}_i = (w_i)^h, 
\, i \in \{ 0, \ldots, N\}$,
 and $u_i$ is defined as in Lemma \ref{R4.2} with partition as in 
\eqref{E4.1third}.

Let $t \in ]0,T]$ and  
$m$ be an integer such that $t \in ]\frac{(m-1)T}{N}, \frac{m T}{N}]$.
 Summing on $i = 0, \cdots, m$, we get


$$ \int _\R \zeta (x) | u^{h}_m(x) | dx \leq \int_\R 
\zeta (x) \vert u^{h}_0(x) \vert dx
 +  c \Vert \zeta''' \Vert_\infty
   \vert h \vert \varepsilon \sum_{i=1}^m \int_\R  |u_i (x)| dx.
$$
Setting $u^{\varepsilon, h } := (u^{\varepsilon})^h$ we obtain
\[\int _\R \zeta (x) | u^{\varepsilon, h }(t,x) | dx \leq \int _\R
| u^{h}_0(x)| \zeta (x) dx
+  c \Vert \zeta''' \Vert_\infty
  \vert h \vert \int _0^T ds \int _\R |u^ \varepsilon (s,x)| dx.\]

So, letting $\varepsilon \to 0$ and using \eqref{4.50} we get
\[\int _\R \zeta  (x) |u^h (t,x)| dx 
\leq \int _\R | (u_0) ^h (x) |\zeta (x) dx
+  c \Vert \zeta''' \Vert_\infty  
   \vert h \vert
 \int _0^T ds \int _\R | u (s,x)| dx
\]
and so \eqref{EBV10}.
Therefore,
\begin{equation}\label{4.51}
   \limsup_{h \to 0} \frac{1}{\vert h \vert}  \int _\R \zeta (x) | u^h (t,x)| dx
   \leq 2 \Vert \zeta \Vert_\infty \vert h \vert
 \norm{u_0 }_{\text{var}} +  c \Vert \zeta''' \Vert_\infty \int _{[0,T] \times \R} 
\vert u(s,x) \vert ds dx,
\end{equation}
where $\Vert \cdot \Vert_{\rm var}$ denotes the total variation.

We denote  the right hand-side of \eqref{4.51} by $\mathcal C(\zeta)$.
 Let $K > 0, \ \varphi \in C_0^\infty(\R)$ 
such that ${\rm supp} \varphi  \subset ]-K,K[,$ and
$ t \in [0,T]$.
Taking $\zeta \equiv 1$ on $]-K,K[$,
 we can replace $\varphi$ with $\varphi \zeta$. Then
\begin{align*}
 \left| \int _\R u(t,x) 
\frac{ \varphi (x) -  \varphi (x-h)}{h} dx \right|
 &= \left|\int _\R \frac{u^h (t,x)}{h} \zeta (x)\varphi (x) dx\right|\\
 \leq &  \frac{1}{\vert h \vert} \Vert \varphi \Vert_\infty 
 \int _\R  \zeta(x) | u^{h}(t,x)| dx.
\end{align*}
So taking the limsup and using \eqref{4.51} we obtain
\[\left | \int _\R u(t,x) \varphi' (x)dx \right|
 \leq \norm {\varphi} _\infty \shc (\zeta)\]
Hence $u(t,\cdot)$ has locally bounded variation on 
$]-K,K[$ and the assertion follows.
\end{prooff}

We now show that, without particular assumptions on the initial conditions,
in the degenerate case, a suitable ``section'' of $\Phi(u(t,\cdot))$
 has at most countably many discontinuities
if so has $u(t,\cdot)$.
We again consider  equation \eqref{PME}  in the sense of distributions
\[\begin{cases}
   \partial _t u = \frac 12 \eta _u '',\quad \eta_u \in \beta (u)\\
   u(0,\cdot ) = u_0\in L^1 \cap L^\infty.
  \end{cases}
\]
We recall that by Proposition \ref{P4.3} a), 
$\eta_u(t,\cdot) \in H^1(\R)$
for a.e.  $t \in ]0,T]$, hence has an absolutely continuous 
version, which will be still denoted by $ \eta_u(t,\cdot)$.
Likewise, since $u(t, \cdot) \ge 0$ a.e., for $\forall t \in [0,T]$,
we shall take a version which is nonnegative everywhere, which
will be still denoted by $u(t, \cdot)$ below.

 Define
\begin{equation} \label{(4.49)prime}
\chi _u = \sqrt {\frac{ \eta_u}{u}} 1_{\{\vert u \vert >0 \}}.
\end{equation}
Here we recall that $u \eta_u \ge 0$, hence
$\frac{\eta_u}{u} \ge 0$ on $\{\vert u \vert > 0 \}$, and that
$\chi_u$ is bounded
by Hypothesis \ref{H3.0}.

\begin{lemma}\label{L4.10}
 Suppose $\beta$ is degenerate,
 let  $t \in [0,T[$ such that $\eta_u(t,\cdot) \in H^1(\R)$
and $ x \in \R$.
If $u(t,\cdot)$ is continuous in $x$, then so is
$\chi_u(t,\cdot).$
In particular,  
$\chi_u(t,\cdot)$ has at most countably many discontinuities
 if so has $u(t,\cdot)$. 
\end{lemma}
\begin{proof}\ 
 It is enough to show that $\chi ^2 _u(t,\cdot)$ 
is continuous in $x$. 
Let $  x_n  \in \R, n \in \N,$ converge to $x$.
We have 
\[\chi_u ^2 (t, x_n ) = \begin{cases}
                        \frac{ \eta_u (t, x_n )}
{u(t, x_n)}, &{\rm if} \quad u (t, x_n )> 0\\
                        0, & {\rm if} \quad u(t,x_n)=0.
                       \end{cases}
\]
\begin{itemize}
   \item If $u(t,x) > 0$, then
   \[\chi_u ^2 (t,x_n) \to \frac{ \eta_u(t,x)}
{ u(t,x)}  = \chi ^2_u (t,x).\]
   \item If $u(t,x) = 0$ then, since $\beta$ is degenerate,
   \[\chi ^2_u (t,x_n) \stackrel{n\to \infty}{\longrightarrow } 0
= \chi ^2_u (t,x) .\]
\end{itemize}
\end{proof}
We have observed that for a
   relatively general coefficient $\beta$, but with a restriction on
 the initial condition,
   $u(t,\cdot)$ (and therefore a suitable section of 
 $\Phi(u(t,\cdot))$) is a.e. continuous, for a.e. $ t \in [0,T]$, 
see Proposition \ref{4.8}.
   We now provide some  conditions on $\beta$ (degenerate) for which
 a suitable section of
   $\Phi(u(t,\cdot))$ is continuous for any initial condition in $L^2 (\R)$.
This will prepare the third main result of this section,
crucially to be used in the next section.

   Let $(u, \eta_u)$ be as usual the solution to \eqref{PME}
and $\chi_u$ as in \eqref{(4.49)prime}.


   \begin{defi}\label{D4.11}
   We say that $\beta$ is {\bf strictly increasing after some zero} 
if there is $e_c \ge 0$ such that
   \begin{enumerate}[i)]
      \item $\beta | _{[0,e_c[}=0$.
      \item $\beta$ is strictly increasing on $[e_c,\infty[$.
\item  If  $e_c = 0$, then  $ \lim_{u \rightarrow 0_+} \Phi(u) = 0  $.
\end{enumerate}
   \end{defi}

\begin{rem}\label{R4.12}
 \begin{enumerate}
\item Condition iii) guarantees that $\beta$ is degenerate.
\item A typical example of a function that is strictly increasing 
after some zero  is given by
\[\beta(u) = u H(u - e_c ) ,\]
where $e_c> 0$ and $H$ is the Heaviside function,
i.e. 
$$
H(u-e_c)  = \left \{
\begin{array}{ccc}
0    &:& u < e_c  \\
{[0,1]} &:& u = e_c \\
1  &: & u > e_c
\end{array}
\right.
 $$
\item We recall that for almost all $t\in ]0,T]$, $\eta_u(t,\cdot)$ 
is continuous.
This will constitute the main ingredient in the proof of the proposition below.
\item 
Suppose that $\beta$ is as in Definition \ref{D4.11}. Then
$ \beta^{-1}$
is single-valued and continuous on $]0,\infty[$.
\end{enumerate}
\end{rem}

\begin{prop}\label{P4.13}
   Suppose $\beta$ strictly increasing after some zero. Then
 for almost all $t \in ]0,T[$, 
  $\chi_u(t,\cdot)$ is continuous.
\end{prop}
\begin{proof} \
We first recall that by Corollary \ref{R4.11},
$u(t, \cdot) \ge 0$ a.e. for all $t \in [0,T]$.

 Let $e_c$ be as in Definition \ref{D4.11}.   
Let $t \in ]0,T]$ for which $\eta_u(t,\cdot)$ is continuous.
   Let $(x_n)$ be a sequence converging to some $x_0 \in \R$. 
The principle is to find a
   subsequence $(n_k)$ such that $\chi_u^2 (t, x_{n_k}) \rightarrow 
\chi_u^2 (t, x_{0})   $.
   In the sequel of the proof, we will omit $t$ and denote
   the functions $u(t,x)$ (resp. $\eta_u(t,x), \chi_u(t,x)$)
by $u(x)$
   (resp.  $\eta_u(x), \chi _u(x)$).

   We distinguish several cases
   \begin{enumerate}
      \item \underline{$u(x_0)\in [0, e_c[$.}
Then $e_c > 0$ and  $\eta_u(x_0) \in \beta (u(x_0)) = 0$. \\Hence
 $\chi _u(x_0)=0$.
      \begin{itemize}
         \item If $u(x_{n_k})< e_c$ for some subsequence $(n_k)$, then
         $$\chi^2_u(x_{n_k}) \equiv  0
         \stackrel{k \to \infty}{\longrightarrow} 0.$$
         \item If there is a subsequence $(n_k)$ such
 that $u(x_{n_k}) \geq e_c $, then
         \[\chi ^2 _u(x_{n_k}) = 
\frac {\eta_u(x_{n_k}) }{u(x_{n_k})} \leq \frac{\eta _u(x_{n_k})}{e_c}
         \stackrel{k \rightarrow \infty}{\longrightarrow} 
\frac{\eta_u(x_0)}{e_c} 
 =0.\]
      \end{itemize}
      \item We suppose now \underline{$u(x_0) \in ]e_c, \infty[$}.

Since $\beta^{-1}$ is single-valued, continuous on
 $]0,\infty[$ and
$\eta_u(x_0) \in \beta (u(x_0)) $, so  
 $\eta_u(x_0)> 0$, we have
\begin{align*}
u(x_0) &= \beta ^{-1}(\eta_u(x_0)) = \beta^{-1}
 (\lim_{n \to \infty} \eta_u(x_n))\\
&= \lim_{n \to \infty} \beta^{-1}(\eta_u(x_n)) = \lim_{n \to \infty} u(x_n).
\end{align*}
Consequently
\[\chi^2 _u (x_n) \rightarrow \chi_u^2 (x_0).\]
\item
\underline{$u(x_0) = e_c $.}

Clearly there are three possibilities.
\begin{enumerate}
   \item there is a subsequence $(n_k)$ with $u(x_{n_k}) \in ]e_c, \infty[$,
   \item there is a subsequence $(n_k)$ with $u(x_{n_k}) \in [0, e_c[$,
   \item there is a subsequence $(n_k)$ with $u(x_{n_k}) = 
e_c \; \forall k \in \N$.
\end{enumerate}

   {\it Case} (a). First we suppose $e_c > 0$.
We have $\eta_u(x_{n_k}) \rightarrow \eta _u(x_0)$.
If $\eta _u(x_0) = 0$ then
  \begin{align*}
\chi_u^2 (x_{n_k}) =\frac{ \eta _u (x_{n_k})}{u(x_{n_k})} 
\rightarrow  0 = \frac{\eta_u(x_0)}{u(x_{0})}  = 
\chi^2 _u (x_0).
   \end{align*}
If $\eta _u(x_0) \neq 0$ then the continuity of $\beta^{-1}$
implies
   \[u(x_{n_k}) = \beta ^{-1} (\eta_u(x_{n_k})) \rightarrow 
\beta ^{-1} (\eta_u(x_0)) = u(x_0) = e_c,\]
   so $\chi^2 _u (x_{n_k}) \rightarrow \chi ^2 _u (x_0)$.\\
   If $e_c = 0$, the result follows since $\beta$ is degenerate.

  {\it Case} (b). In this case $e_c$ is again strictly positive.
 Since 
$\eta_u(x_{n_k}) \in \beta(u(x_{n_k}) = 0$ we have
$\chi_u(x_{n_k})=0$, 
hence $\chi_u(x_{n_k}) \stackrel{k \rightarrow \infty}\longrightarrow 0$.
   But $0 = \eta_u(x_{n_k}) \stackrel{k \rightarrow \infty}\longrightarrow 
\eta_u(x_0)$.
   This implies that $\eta _u(x_0)=0$, so $\chi ^2 _u (x_0)=0$.
   
   {\it Case} (c). We have $u(x_{n_k}) = e_c$.
If $e_c = 0$ the result follows trivially by definition of $\chi_u$.
Therefore we can suppose again that $e_c > 0$. Then 
$u(x_{n_k}) = e_c  \stackrel{k \rightarrow \infty} \longrightarrow e_c$, so
   \begin{align*}
\chi_u^2 (x_{n_k}) =\frac{ \eta _u (x_{n_k})}{e_c} 
\rightarrow \frac{\eta_u(x_0)}{e_c} =
 \frac{\eta_u(x_0)}{u(x_0)} =
\chi^2 _u (x_0).
   \end{align*}

   This completes the proof.
   \end{enumerate}
\end{proof}

\section{The probabilistic representation of
 the deterministic equation}

\setcounter{equation}{0}

We again consider  the  $\beta: \R \rightarrow 2^\R$ 
satisfying Hypothesis \ref{H3.0}.
 We aim at providing a probabilistic 
representation for solutions to equation \eqref{PME}.
 Let $ u_0 \ge 0$ such that 
$\int_\R u_0(x) dx = 1$ and $u_0 \in  L^\infty (\R)$.

We consider a multi-valued map 
$\Phi: \R \rightarrow 2^{\R_+}$ such that 
$$ \beta(u) = \Phi^2(u) u, \quad u \in \R, $$
which is bounded, i.e. 
$$ \sup_{u \in \R} \sup \Phi(u) < \infty. $$

The degenerate case is much more difficult than the non-degenerate case
which was solved in \cite{BRR}.

\begin{defi} \label{D51}
Let $(u,\eta_u)$ be the solutions in the sense of Proposition
\ref{R4.1} to
          equation \eqref{PME}.
 i.e.
\begin{equation}
\label{5.1}
\left \{
\begin{array}{ccc}
\partial_t u&=& \frac{1}{2} \partial_{xx}^2 (\eta_u),  
\quad {\rm on} \
L^1 (\R) 
\\
u(0,x)& =& u_0(x).  
\end{array}
\right.
\end{equation}
\end{defi} 
We say that \eqref{PME}
has a    {\bf probabilistic representation}, if there is
a filtered probability space $(\Omega, \shf, P, (\shf_t))$,
 an $ (\shf_t))$-Wiener process $W$ and,
at least one
process $Y$,
such there exists
$\chi_u \in  (L^1 \bigcap L^\infty)([0,T] \times \R) $ with 
\begin{equation}
\label{EProbDeg}
\left \{
\begin{array}{ccc}
Y_t &=& Y_0 + \int_0^t \chi_u(s,Y_s)) dW_s  \ {\rm in \ law} \\
\chi_u(t,x) &\in& \Phi(u(t,x)) \ {\rm for} \ dt \otimes dx \ {\rm a.e.}
  (t,x) \in [0,T] \times \R,   \\
{\rm Law \quad density } (Y_t) &=& u(t,\cdot). \\
u(0,\cdot) &=& u_0.
\end{array} \right.
\end{equation}

We recall the main result of \cite{BRR}, Theorem 4.3.

\begin{theo} \label{P31}
When $\beta$ is non-degenerate then \eqref{PME} has a
 probabilistic representation, with 
 $$\chi _u = \sqrt {\frac{ \eta_u}{u}} 1_{\{\vert u \vert > 0 \}}.
$$
\end{theo}
\begin{rem} \label{R54}
In the non-degenerate case the representation is unique.
\end{rem}

We will show that, even in the degenerate case, \eqref{PME}
 has a probabilistic representation.
\begin{theo} \label{T52} Suppose that $\beta$ is
degenerate.
  Then equation \eqref{PME}  admits a probabilistic representation
 if one of the following conditions are
verified.
\begin{enumerate}
\item $\beta$ is strictly increasing after some (non-negative) zero.
\item $u_0$ has locally bounded variation. 
\end{enumerate}
\end{theo}

\begin{proof}

We will make use of Theorem \ref{P31}.
Let $\varepsilon \in  ]0,1] $ and set 
$$ \Phi_\varepsilon (u) =
  \sqrt{\Phi^2 (u) + \varepsilon},
 \  \beta^\varepsilon (u) = \beta(u) + \varepsilon u. $$

Let $(u^{(\varepsilon)}, \eta_{u^{(\varepsilon)}})$
 the solution 
 to the 
deterministic PDE  (\ref{PME}), with
$\beta^\varepsilon$ replacing $\beta$.
Define
 \begin{equation}\label{(5.3)bis} 
\chi^\varepsilon = 
\sqrt {\frac{\eta_{u^{(\varepsilon)}}}{u^{(\varepsilon)} }} 
1_{\{\vert u^{(\varepsilon)}  \vert > 0 \}}.
\end{equation}
We note that since $\Phi_\varepsilon, \varepsilon \in ]0,1]$
are uniformly bounded, so are $\chi^\varepsilon, \varepsilon \in ]0,1]$.


By Theorem \ref{P31}, there exists a
unique solution $Y = Y^\varepsilon$ in law of 
\begin{equation}
\label{EProbDeg1}
\left \{
\begin{array}{ccc}
Y_t &=& Y_0 + \int_0^t \chi^\varepsilon(s,Y_s)) dW_s  \\
\chi^\varepsilon (t,x) &\in& \Phi_\varepsilon (u^{(\varepsilon)}(t,x)) 
\ {\rm for} \ dt \otimes dx \ {\rm a.e.}
  (t,x) \in [0,T] \times \R.   \\
{\rm Law \quad density } (Y_t) &=& u^{(\varepsilon)}(t,\cdot) \\
u^{(\varepsilon)} (0,\cdot) &=& u_0.
\end{array} \right.
\end{equation}

Since $\Phi$ is bounded, using the Burkholder-Davies-Gundy
 inequality 
one obtains 
\begin{equation} \label{(5.4)prime}
 \E \vert Y^\varepsilon_t - Y^\varepsilon_s \vert^4 \le {\rm const.}
(t-s)^2.
\end{equation}
This implies (see for instance \cite{ks} Problem 4.11 of Section 2.4)
that the laws of $Y^\varepsilon, \varepsilon  > 0$   are tight.
Consequently, there is a subsequence $Y^n := Y^{\varepsilon_n}$
 converging in law
(as $C[0,T]$-valued random elements)  to some process $ Y$.
We set $u^n := u^{(\varepsilon_n)}$, where we recall that
 $u^n(t, \cdot) $ is  the law of $Y^n_t$,
and $\chi^n := \chi^{\varepsilon_n}$.

Since 
$$ [Y^n]_t = \int_0^t (\chi^{n})^2 (s, Y^n_s) ds, $$
and $E([Y^n]_T)$ is finite, $\Phi$ being bounded, 
the continuous local martingales $Y^n$ are indeed  martingales. 

By  Skorokhod's theorem  there is a new  probability space
$(\Omega, \shf,  P)$ 
and processes $\tilde Y^n$, with the same distribution
as $Y^n$ so that $\tilde Y^n$ converge to some process
 $\tilde Y$,  distributed as $Y$,  as $C([0,T])$- random elements
$P$-a.s.
In particular, those processes $\tilde Y^n$ remain martingales
with respect to the filtrations generated by them.
We denote   the sequence   $\tilde Y^n$
(resp. $ \tilde Y$), again by $Y^n$ (resp. $Y$).

\begin{rem}\label{RUY}
We observe that, for each $t \in [0,T]$,
$u(t, \cdot)$ is the law density of $Y_t$.
In fact, for any $t \in [0,T]$, $Y^n_t$ converges in probability
to $Y_t$; on the other hand $u^n(t, \cdot)$, which is the law of $u^n_t$
 converges to $u(t, \cdot)$  in $L^1(\R)$, by Proposition \ref{R4.1} 5.
\end{rem}

\begin{rem} \label{Rem4.10}
Let $\shy^n$ (resp.  $\shy$)
be the canonical filtration associated with $Y^n$ (resp. $Y$).

 We set 
$$ W^n_t = \int_0^t \frac{1}{\chi^{n}} (s, Y^n_s) dY^n_s.$$
Those processes $W^n$ are standard $(\shy^n_t$) -Wiener processes 
since $[W^n]_t =
t$ and because of L\'evy's characterization theorem of Brownian motion.
Then one has
$$  Y^n_t = Y^n_0 + \int_0^t \chi^{n}(s, Y^n_s) dW^n_s.$$
\end{rem}
\end{proof}



We aim to prove first that 
\begin{equation}\label{EFinDeg}
 Y_t = Y_0 + \int_0^t  \chi_u(s, Y_s) dW_s.
\end{equation}
where $\chi_u$ is defined as in \eqref{(4.49)prime}.
Once this equation is established for the given $u$,
the statement of Theorem \ref{T52} would be completely
proven because of Remark \ref{RUY}.
In fact, that remark shows  in particular  the third line of (\ref{EProbDeg}).

Taking into account, Theorem 4.2 of Ch. 3 of \cite{ks}, 
 to establish (\ref{EFinDeg}),
 it will be enough to prove that $Y$ is a $\shy$-
martingale with quadratic variation
$[Y]_t = \int_0^t \chi_u^2(s, Y_s) ds. $

Let $s, t \in [0,T]$ with  $t > s$  and 
  $\Theta$  a bounded continuous function from
$C([0,s]) $ to $\R$.

 In order to prove the martingale property for $Y$, we need to show that
$$ E\left((Y_t - Y_s) \Theta(Y_r, r \le s) \right) = 0.$$
This follows by \eqref{(5.4)prime} because $Y^n \rightarrow Y$ a.s. as 
$C([0,T])$-valued process and 
 $$ E \left((Y^n_t - Y^n_s) \Theta(Y^n_r, r \le s) \right) = 0.$$

It remains to show that $Y_t^2 - \int_0^t \chi_u^2(s, Y_s)  ds,
  t \in [0,T]$,  defines a $\shy$-martingale, which in turn follows, if 
for $ t > s$ we can  verify 
$$  E\left((Y_t^2 - Y_s^2 - \int_s^t  \chi_u^2(r, Y_r) dr )
\Theta(Y_r, r \le s) \right) = 0.$$
The left-hand side decomposes into $ I^1(n) + I^2(n) + I^3(n) $ 
where 
\begin{eqnarray*}
 I^1(n) &=&  E \left(  (Y_t^2 - Y_s^2 - \int_s^t \chi_u^2(r, Y_r) dr)
  \Theta(Y_r,
r \le s) \right ) \\ & - &  E \left ( \left((Y^n_t)^2 - (Y^n_s)^2 - \int_s^t
\chi_u^2(r, Y^n_r)  dr \right) 
 \Theta(Y^n_r, r \le s) \right), \\
 I^2(n) &=&  E\left(\left( (Y^n_t)^2 - (Y^n_s)^2 - \int_s^t  
 \chi^n (r, Y^n_r)^2   dr
  \right) \Theta(Y^n_r, r \le s)
\right ), 
\end{eqnarray*}
and $$
 I^3(n) =  E \left( \int_s^t \left(  \chi^n 
(r, Y^n_r)^2 -  \chi_u^2(r, Y^n_r)   \right) dr \Theta(Y^n_r, r \le s)
\right ). $$

We start showing the convergence of 
$I^3(n)$.
 Now $\Theta(Y^n_r, r \le s)$ converges a.s. to $\Theta(Y_r, r \le s)
 $ and it is dominated by a constant.
so that  it suffices to consider the
 expectation of 
$$  \int_s^t \left \vert  \chi^n 
(r, Y^n_r)^2  - \chi_u (r, Y^n_r)^2  \right \vert  dr  $$

which is equal to
$$ I(n) =\int_s^t dr \int_\R   
\left \vert  \eta_{u^{(\varepsilon_n})} 
(r, y) - \chi_u (r, y)^2 u^n (r,y)  \right \vert 
 dy. 
$$

By Proposition \ref{5.3} below $\eta_{u^{(\varepsilon)}}
\rightarrow \eta_u$ in $L^1([0,T]\times \R)$ as 
$\varepsilon \rightarrow 0$. Furthermore,
Proposition \ref{R4.1} 5),  see also the  
theorem in the introduction of \cite{BeC81},
implies that $u^\varepsilon(t,\cdot)$ converges to $u(t, \cdot)$
in $L^1(\R)$, as $\varepsilon \rightarrow 0$, uniformly in $t \in [0,T]$.
Hence  Lebesgue's dominated convergence theorem implies that
$I(n) \rightarrow 0$, since $\chi_u$ is bounded.

We go on with the analysis of $I^2(n)$ and $I^1(n)$. 
$I^2(n)$  equals to zero because
$Y^n$  is a martingale with quadratic variation given by
$$[Y^n]_t = \int_0^t \chi^n(r, Y^n_r)^2 dr.$$

We  finally treat $I^1(n)$. We recall that $Y^n \rightarrow
Y$ a. s. as  random elements in $C([0,T])$
and that the sequence $E\left
((Y^n_t)^4\right ), $ is 
bounded, so $(Y^n_t)^2$   are uniformly integrable. Therefore,
for $t > s $ we have 
$$ E \left( (Y^n_t)^2 - (Y^n_s)^2) \Theta(Y^n_r, r \le s) \right ) -
E \left((Y_t^2 - Y_s^2) \Theta(Y_r, r \le s) \right) \rightarrow 0,$$
when $n \rightarrow \infty$.
It remains to prove that
\begin{equation} \label{EDec}
 E \left( \int_s^t   \chi_u^2(r, Y_r) dr \Theta(Y_r, r \le s)     - 
 \int_s^t \chi_u^2(r, Y^n_r) dr  \Theta(Y^n_r, r \le s)
\right) \rightarrow 0.
\end{equation}
Under the assumptions of the theorem,  for fixed $r \in [0,T]$, 
by the second and third main results of Section 4 
(see Propositions \ref{P4.8}, \ref{P4.13} and Remark \ref{R4.9}),
 $\chi_u(r, \cdot)$ has at most a countable number of discontinuities.
Moreover,  the law of $Y_r$  has a density and it is therefore 
non atomic. So, let $N(r)$ be the null event of $\omega \in \Omega$  such that 
$Y_r (\omega)$ is a point of discontinuity of $\chi_u(r, \cdot)$.
For $\omega \notin N(r)$ we have 
$$ \lim_{n \rightarrow \infty} \chi_u^2(r, Y_r^n (\omega)) = 
 \chi_u^2(r, Y_r (\omega)).$$
Now, Lebesgue's dominated convergence and Fubini's
 theorem imply (\ref{EDec}).
So equation  (\ref{EFinDeg}) is shown.

It remains to prove the following result which is based
on our first main result of Section 4, see Theorem \ref{T4.5}.
\begin{prop}\label{5.3} 
Let $\eta_u^\varepsilon  := \eta_{u^{(\varepsilon)}}, \varepsilon > 0 $.
Then  $ \eta_u^\varepsilon \rightarrow \eta_u$  in $L^1([0,T]\times \R)$
as $\varepsilon \rightarrow 0$. 
\end{prop}
\begin{proof}\ 
   We set 
   \[j _\varepsilon (x) := \int _0^x \beta _\varepsilon^\circ 
 (y)dy = j(x) + \varepsilon \frac {x^2}2.\]
where $\beta^\circ_\varepsilon$ is the minimal section of 
$\beta^\varepsilon$ and clearly
$    \beta^\circ_\varepsilon(x) = \beta^\circ (x) + \varepsilon x, 
x \in \R$.
According to Theorem \ref{T4.5} we have
   \begin{equation}\label{5.13}
      \int _\R \; j _\varepsilon (u^{(\varepsilon)} (T, x)) dx + 
\frac 12 \int _0^T ds
 \int _\R {(\eta ^\varepsilon _u)'}^2 (s, x) dx = \int _\R j (u_0(x)) dx.
   \end{equation}
   In particular, 
   \begin{equation}\label{5.14}
      \sup_{\varepsilon > 0} \int _0^T ds
 \int _\R {(\eta _u^\varepsilon )'}^2 (s,x) dx 
\leq \int _\R j (u_0(x)) dx< \infty.
   \end{equation}
   So, by \eqref{(E4.1second)}, 
 the family $\{\eta_u ^\varepsilon, \varepsilon \in ]0,1]\}$ 
is weakly relatively compact in $L^2 ([0,T]; H^1(\R))$,
hence also in 
$L^2 ([0,T]; L^2 (\R)) = L^2 ([0,T]\times \R)$.
 We recall that
 $u^{(\varepsilon)} (t, \cdot)\rightarrow u(t,\cdot)$ in $L^1 (\R)$
uniformly in $t \in [0,T]$.

      Let $(\varepsilon_n)$ be a sequence converging to zero. 
There is a subsequence $(n_k)$ such that
      $\eta^k_u := \eta_u ^{\varepsilon_{n_k}}$ converges weakly
 in $L^2 ([0,T]\times \R)$ to some
      $\xi \in L^2 ([0,T]\times \R)$. For any $\alpha \in C_0^\infty(\R)$,
      \[\int_\R u^{(\varepsilon)} (t,x) \alpha (x)dx = \int _\R u_0 (x)
      \alpha(x) dx + \frac 12 \int _0^t ds
 \int _\R \eta^\varepsilon_u  (s,x) \alpha'' (x) dx.\]
      Taking the limit when $k\rightarrow \infty$, we get
      \begin{equation}\label{5.11}
         \int _\R u(t,x) \alpha (x) dx = \int _\R u_0(x) \alpha (x) dx 
+ \frac 12 \int _0^t ds \int _\R \xi(s, x) \alpha '' (x) dx.
      \end{equation} 
Let $K > 0$. Since $\beta $ is maximal monotone,
 $v\mapsto \beta(v)$ is
a maximal monotone map from $L^2 (\R \times [-K,K])$ to 
$L^2 (\R \times [-K,K])$. Therefore,
  \cite{Barbu2}, p.37,
Proposition 1.1 (i) and (ii), imply that
this map is weakly-strongly closed.
Since, by \eqref{E4.1second},  $u^{(\varepsilon)}$ converges to $u$
in $L^2 ([0,T] \times [-K,K])$, it follows that $\xi \in \beta(u)$
a.e. on $[0,T] \times [-K,K]$ for all $K > 0$,
so, $\xi \in \beta(u)$  a.e. 
 By the uniqueness of \eqref{PME}
we get $\xi = \eta _u$ a.e.

Let $\varepsilon_n   \rightarrow 0$.
The rest of the paper will be devoted to the proof of the existence
 of a subsequence $(\eta_u^k) :=  \eta_{u^{(\varepsilon_{n_k})}}$  
converging (strongly)  to $\eta_u$
in $L^2 _{\text{loc}}([0,T]) \times \R)$.
Since $\eta_u^k \in \beta(u^{(\varepsilon_{n_k})})$, we have
$$ \vert \eta_u^k \vert \le (c + \varepsilon_{n_k}) 
\vert u^{(\varepsilon_{n_k})} \vert. $$
Hence $\{\eta_u^k \} $ is equintegrable on $[0,T] \times \R$.
Therefore, the existence of such a subsequence   
 completes
the proof.

We will need the following well-known lemma.
\begin{lemma}\label{L5.7}
 Let $H$ be a Hilbert space, $(f_n)$ be a sequence in $H$ converging
 weakly to some $f\in H$.
 Suppose
 \[\limsup_{n \rightarrow \infty} \norm{f_n}^2 \leq \norm f ^2.\]
 Then $f_n \rightarrow f$ strongly in $H$.
\end{lemma}
We apply the previous Lemma to establish the existence of a subsequence still denoted by
$(\eta^k_u)$ such that $(\eta^k _u)'$ converges strongly to $\eta' _u$ in $L^2([0,T]\times \R)$.
For this, we will prove that
\[\limsup_{k \rightarrow \infty} \int _0^T ds \int _\R dx \; {(\eta_u^k)'}^2 (s,x) dx
\leq \int _0^T ds \int _\R dx \; {\eta_u'}^2 (s,x) dx .\]
We consider \eqref{5.13} for $\varepsilon = \varepsilon _{n_k}$ and we let $k$ go to infinity. First,
for $t\in [0,T]$ we have
\begin{equation}\label{5.21}
   \int_\R dx j _\varepsilon (u^{(\varepsilon)} (t,x)) =
 \int _\R dx \; j (u^{(\varepsilon)} (t,x))
   + \varepsilon \int _\R (u^{(\varepsilon)} )^2 (t,x) dx.
\end{equation}
Since $j$ is continuous,
\[\lim_{k \rightarrow \infty} j (u^{(\varepsilon_{n_k})}(t,x) ) dx =
 j (u(t,x))\quad \text{a.e.}.\]
By Fatou's lemma
\begin{equation}\label{5.22}
 \int _\R j( u(t,x)) dx \leq
 \liminf_{k\rightarrow \infty}
 \int  _\R j (u^{(\varepsilon _{n_k})}(t,x)) dx.
\end{equation}
Again Theorem \ref{T4.5} implies
\[\int _\R j (u (T,x) ) dx + \frac 12 \int _0^T ds
 \int _\R {(\eta_u)' }^2 (s,x) dx = \int _\R j (u_0 (x)) dx.\]
This together with \eqref{5.13} gives
\begin{eqnarray}\label{5.23}
 \frac 12  \int _0^T ds \int _\R {(\eta_u^k)'} ^2 ( s, x) dx &=&
\frac 12 \int _0^T  ds  \int _\R {\eta_u'}^2 (s,x) dx
   + \int _\R j(u(T,x)) dx  \nonumber\\
&&\\
   &-& \int _\R j ( u^{(\varepsilon _{n_k})} (T,x)) dx -
\frac {\varepsilon_{n_k}} 2 \int_\R u^{(\varepsilon _{n_k})} (T, x) ^2 dx.
 \nonumber
\end{eqnarray} 
Since by Corollary \ref{R4.4}
\[\int _\R u^{(\varepsilon)} (T,x)^2 dx \leq \int _\R u_0 (x) ^2 dx,\]
the last term in \eqref{5.23} 
converges to zero when $k \rightarrow \infty$. Taking the limsup when $k\rightarrow \infty$
in \eqref{5.23} and using \eqref{5.22} we get
$$
 \limsup_{k \rightarrow \infty}   \int _0^T ds \int _\R (\eta^k _u)'^2
  (s,x) dx \le 
 \int _0^T ds \int _\R {\eta_u'}^2(s,x) dx.
$$
Consequently, by Lemma \ref{L5.7}
\begin{equation} \label{(5.13)prime}
\lim_{k \rightarrow \infty} \int _0^T ds \int _\R ds
 ((\eta^k _u)'(s,x) - \eta_u' (s,x))^2 =0.
\end{equation}
Now let us
 finally  prove that $\eta^k_u \rightarrow \eta_u$ (strongly) in
$L^2 _{\rm loc} ([0.T],\times \R)$.
Let $x\in \R$. We recall that $\eta_u, \eta^k_u(t,\cdot)$ vanish at 
infinity since they belong to
$H^1(\R)= H_0^1(\R)$. So we can write, for $x \in \R$,
\begin{align*}
 (\eta_u ^k (t,x) - \eta_u(t,x)) ^2 &= 2 \int_{-\infty}^x (\eta^k_u )'
 (t, y) - \eta_u ' (t,y)
 (\eta^k_u (t,y) - \eta_u (t,y)) dy\\
 &\leq 2 \left\{ \int_{-\infty}^x ((\eta^k_u)' - \eta_u') ^2 (t,y) dy
 \int _{-\infty}^x (\eta_u^k -\eta_u) ^2
 (t,y)dy\right\}^{\frac 12}.
\end{align*}
Integrating from $0$ to $T$, by the Cauchy-Schwarz inequality, the quantity
$$\int _0^T dt ({\eta_u}^k-\eta_u )^2 (t,x)$$
is bounded by
 $$ 2 \sqrt{\int _0^T dt \int _{\R } ({\eta^k_u}' - \eta_u ')^2 (t,y) dy}
\sqrt{\int _0^Tdt \int _{\R } ((\eta^k_u) - \eta_u )^2 (t,y) dy}.
$$
On the other hand, using Corollary \ref{R4.4} and \eqref{E3.0},
we have 
$$
 \int _0^T ds \int _\R dy \; \eta^k_u (t,y)^2
 \leq
 \const\int _0^T ds \int _\R dy \; u^{(\varepsilon_{n_k})}(t,y)^2
\leq \const  T \norm{u_0}_{L^2}^2
$$
and likewise
$$  \int _0^T ds \int _\R dy \; \eta_u^2
 (t,y) \leq  \const. T \norm{u_0}^2_{L^2}.
$$
Consequently, maybe with another const,
 \[\sup_{x \in \R} \int _0^T dt (\eta_u ^k - \eta_u ) ^2 (t,x) \leq
  \sqrt {T} \const \norm{u_0}_{L^2}
 \norm{{\eta_u^k}' - \eta_u '} _{L^2 ([0,T]\times \R)},\]
which by \eqref{(5.13)prime} converges to zero. 
\end{proof}



{\bf ACKNOWLEDGEMENTS} 

\noindent
Financial support through the SFB 701 at Bielefeld University and
NSF-Grant 0606615
 is gratefully acknowledged.



\bigskip
\begin{thebibliography}{9}
 

\bibitem{adams} R.A. Adams, {\it Sobolev spaces}, Academic press 1975.

\bibitem{A86} D. G. Aronson, {\it The porous medium equation}, in Lect. Notes Math. Vol. {\bf 1224}, (A. Fasano and al.
 editors), Springer, Berlin, 1--46, 1986.

\bibitem{bak86} P. Bak, {\it How Nature Works: The Science of 
Self-Organized Criticality.} New York: Copernicus, 1986.

\bibitem{BanJa} P. Banta, I.M. Janosi, {\it Avalanche dynamics from
anomalous diffusion.}
Physical review letters {\bf 68}, no. 13, 2058--2061 (1992).

\bibitem{Barbu1} V. Barbu, {\it 
Nonlinear semigroups and differential equations in Banach spaces.}
 Noordhoff International Publishing, Leiden, 1976.

\bibitem{Barbu2} V. Barbu, {\it Analysis and control of nonlinear
 infinite dimensional systems}, Academics  Press, San
Diego, 1993.

\bibitem{BBDRSoc} V. Barbu, Ph. Blanchard, G. Da Prato, M. R\"ockner,
{\it Self-organized criticality via stochastic partial differential
  equations.} http://arxiv.org/abs/0811.2093.


\bibitem{BDPR09} V. Barbu,  G. Da Prato, M. R\"ockner,
{\it Stochastic porous media equations and 
Self-organized criticality.} 
Comm. Math. Phys. {\bf 285},  901--923, 2009.

\bibitem{BRR} Ph, Blanchard, M. R\"ockner, F. Russo.
{\it Probabilistic representation for solutions of an
irregular  porous media  equation.}
BiBoS Bielefeld Preprint, 2008 08-05-293.
http://aps.arxiv.org/abs/0805.2383




\bibitem{BCRV} S. Benachour, Ph.  Chassaing, B. Roynette, P. Vallois,
{\it Processu associ\'es \`a l' \'equation des milieux poreux.}
 Ann. Scuola Norm. Sup. Pisa Cl. Sci. (4) {\bf 23}, no. 4, 793--832 (1996).

\bibitem{BeBrC75} Ph. Benilan, H. Brezis, M. Crandall,
 {\it A semilinear equation in $L^1(\R^N)$}.
Ann. Scuola Norm. Sup. Pisa, Serie IV, II Vol. {\bf 30},  No 2
  523--555  (1975). 

\bibitem{BeC81}  Ph. Benilan, M. Crandall, 
    {\it The continuous dependence on $\varphi$ 
of solutions of $u_t - \Delta \varphi(u) = 0$},
Indiana Univ. Mathematics Journal, Vol. {\bf 30},  No 2
  161--177 (1981). 





\bibitem{BrC79} H. Brezis, M. Crandall,
 {\it Uniqueness of solutions of the initial-value
problem for  $u_t - \Delta \varphi (u) = 0$},
J. Math. Pures Appl. {\bf 58},  
  153--163 (1979). 



 \bibitem{clpvz}  R. Cafiero, V. Loreto, L. Pietronero, A. Vespignani
    and S. Zapperi, 
    {\it Local rigidity and self-organized criticality for avalanches.},
Europhysics Letters, {\bf 29} (2),  
111-116 (1995). 


\bibitem{choquet} G. Choquet,
{\it Lectures on analysis. Vol. II: Representation theory}.
Edited by J. Marsden, T. Lance and S. Gelbart.
W. A. Benjamin, Inc., New York-Amsterdam, 1969.


\bibitem{CE75} M.G. Crandall, L.C. Evans,
 {\it On the relation of the operator $\frac{\partial}{\partial s}
+ \frac{\partial}{\partial \tau}$
to evolution governed by accretive operators},
Israel Journal of Mathematics Vol. {\bf 21},  
No 4, 261--278 (1975). 

\bibitem{DS} N. Dunford and J.T. Schwartz.
 {\it Linear operators}, Part I, General theory.
John Wiley, 1988.


  

\bibitem{E77} L.C. Evans,
 {\it Nonlinear evolution equations in
an arbitrary Banach space},
Israel Journal of Mathematics Vol. {\bf 26},  
No 1,  1--42 (1977). 


\bibitem{E78} L.C. Evans,
 {\it Application of nonlinear semigroup theory
to certain partial differential equations},
M. G. Crandall Ed., Academic Press, NY,    
pp. 163--188, 1978.

\bibitem{graham} 
C. Graham,  Th. G. Kurtz, S.  M\'el\'eard, S., Ph. Protter, 
M. Pulvirenti, D. Talay,
{\it Probabilistic models for nonlinear partial differential equations.}
Lectures given at the 1st Session and Summer School held in
Montecatini  Terme, May 22--30, 1995. Edited by Talay and L. Tubaro.
 Lecture Notes in Mathematics, 1627, Springer-Verlag.

 \bibitem{J00} B. Jourdain, 
{\it Probabilistic approximation for a porous medium equation.}
 Stochastic Process. Appl. {\bf 89}, no. 1, 81--99 (2000).

 \bibitem{ks} I. Karatzas, S.E. Shreve,  {\it 
Brownian motion and calculus}, Springer--Verlag,
Second Edition 1991.




\bibitem{mckean} H.P., Jr. McKean,
{\it Propagation of chaos for a class of non-linear parabolic equations.}
  Stochastic Differential Equations (Lecture Series in Differential 
Equations, Session 7, Catholic Univ., 1967) 
 pp. 41--57. Air Force Office Sci. Res., Arlington, Va.
60.75.

\bibitem{rs} M. Reed, B. Simon, {\it Methods of modern mathematical physics. II. Fourier analysis, 
self-adjointness},  Academic Press, New York-London, 1975.


 
 



 \bibitem{sho97} R.E. Showalter, 
{\it Monotone operators in Banach space and nonlinear partial differential equations}.
 Providence, RI: American Math. Soc., 1997



 \bibitem{stein} E.M. Stein,
{\it Singular integrals and differentiability properties of
  functions.}
 Princeton Mathematical Series, No. 30 Princeton University Press,  1970.

 \bibitem{SV79} D.W. Stroock and S.R.S. Varadhan,  {\it 
Multidimensional Diffusion Processes}, Springer--Verlag, 1979.
 
\bibitem{sznit} A.-S. Sznitman, {\it Topics in propagation of chaos.}
  Ecole d' \'et\'e de Probabilit\'es de Saint-Flour XIX---1989,
  165--251,  Lecture Notes in Math., 1464, Springer, Berlin, 1991. 

\bibitem{Triebel} H. Triebel, {\it Interpolation Theory, Function Spaces,
 Differential
Operators}, North Holland, Amsterdam, 1978.


\bibitem{moreaux} K. Yosida, {\it Functional analysis}.
 Sixth edition, {\bf 123}. Springer-Verlag,  1980.

\end{thebibliography}
 \end{document}